\documentclass[reqno, 12pt, superscriptaddress, twoside, tightenlines, notitlepage]{revtex4-1} 

\usepackage[utf8]{inputenc} 
\usepackage{youngtab}
\usepackage{tikz}
\usepackage{bold-extra}

\usepackage{geometry} 
\usepackage{graphicx} 
\usepackage{amsmath}
\usepackage{amssymb}
\usepackage{amsthm}

\usepackage{booktabs} 
\usepackage{array} 
\usepackage{paralist} 
\usepackage{verbatim} 
\usepackage{subfig} 
\usepackage{chngcntr}

\usepackage{fancyhdr} 
\newtheorem{theorem}{Theorem}
\newtheorem{proposition}{Proposition}
\newtheorem{lemma}{Lemma}
\newtheorem{corollary}{Corollary}
\newtheorem{definition}{Definition}

\begin{document}
\title{Soliton cellular automaton associated with $G_2^{(1)}$ crystal base}
\author{Kailash C. Misra}
\email{misra@unity.ncsu.edu}
\affiliation{
Department of Mathematics\\ North Carolina State University\\ 
Raleigh, NC 27695-8205}
\author{Masato Okado}
\email{okado@sigmath.es.osaka-u.ac.jp}
\affiliation{Department of Mathematical Science\\
Graduate School of Engineering Science\\
Osaka University\\
Toyonaka, Osaka 560-8531, Japan}
\author{Evan A. Wilson}
\email{eawilso4@ncsu.edu}
\affiliation{
Department of Mathematics\\ North Carolina State University\\ 
Raleigh, NC 27695-8205}
\begin{abstract}
\noindent\normalsize{We calculate the combinatorial $R$ matrix for all elements of $\mathcal{B}_l\otimes \mathcal{B}_1$ where $\mathcal{B}_l$ denotes the $G_2^{(1)}$-perfect crystal of level $l$, and then study the soliton cellular automaton constructed from it.  The solitons of length $l$ are identified with elements of the $A_1^{(1)}$-crystal $\tilde{\mathcal{B}}_{3l}$.  The scattering rule for our soliton cellular automaton is identified with the combinatorial $R$ matrix for $A_1^{(1)}$-crystals.}
\end{abstract}

\maketitle

\section{Introduction} A cellular automaton is a dynamical system in which to points in the one-dimensional space lattice are assigned discrete values which evolve according to a deterministic rule.  Soliton cellular automata (SCA) are a  kind of cellular automata which possess stable configurations analogous to solitons in integrable partial differential equations. Solitons move with constant velocity proportional to the length if there is no collision. After collision their lengths are preserved but the phases are shifted. A simple example of SCA is the ``box-ball system''(\cite{TS}, \cite{T}), for which the scattering rule of solitons was shown to be described by the combinatorial $R$ matrix of crystal bases of the quantum group $U_q'(A_{n-1}^{(1)})$.  Crystal bases (\cite{Ka}, \cite{Lu}) are a powerful combinatorial tool in the representation theory of quantum groups, which, when they exist, relate the modules of a quantum group to a combinatorial structure called a crystal.  In \cite{HKT}, a class of SCA was formulated in terms of crystals of affine Lie algebras.  In this formulation the time evolution operator is given by the row-to-row transfer matrix of integrable vertex models at $q=0$, represented as a product of the combinatorial $R$ matrix of crystals.   In \cite{HKOTY}, using this approach, the scattering rules for the SCA associated with the perfect crystals given in \cite{KKM} were determined.  In \cite{FOY} the phase shift of the $U_q'(A_n^{(1)})$ SCA was related to the energy function of $A_{n-1}^{(1)}$-crystals.   In \cite{Yd}, Yamada gave the scattering rules for the SCA associated with the perfect $D_4^{(3)}$-crystals given in \cite{KMOY}, and in \cite{M} those of the SCA associated with the $D_n^{(1)}$-crystal $\mathcal{B}^{n,l}$ were also given recently.

The affine Lie algebra $G_2^{(1)}$ is the infinite dimensional analog of the finite dimensional exceptional Lie algebra $G_2$ (see \cite{K}). For each integer $l\geq 1$ it has a perfect crystal $\mathcal{B}_l=\mathcal{B}^{1,l}$, which is realized as a set of restricted semistandard tableaux with a single row and at most $l$ columns in \cite{Yn}, however the explicit action of the 0 arrows was not given.   In \cite{MMO}, a coordinate construction of $\mathcal{B}_l$ as a subset of  $(\mathbb{Z}/3)^6$ was given along with the explicit action of the 0 arrows.  

In this paper, using the realizations of $\mathcal{B}_l$ given in \cite{Yn} and \cite{MMO} we compute the combinatorial $R$ matrix $\mathcal{B}_l\otimes \mathcal{B}_1 \tilde{\to}\mathcal{B}_1\otimes \mathcal{B}_l$ for $G_2^{(1)}$, then use it to find the scattering rule of the SCA associated with $\mathcal{B}_l$.  We see that solitons of length $l$ are parameterized by the $A_1^{(1)}$-crystal $\tilde{\mathcal{B}}_{3l}$ and that when two solitons collide, the scattering rule is described by the combinatorial $R$ matrix $\tilde{\mathcal{B}}_{3l_1}\otimes \tilde{\mathcal{B}}_{3l_2}\tilde{\to}\tilde{\mathcal{B}}_{3l_2}\otimes \tilde{\mathcal{B}}_{3l_1}, l_1>l_2$. This SCA has two characteristic features. Firstly, the parameterization of length $l$ solitons by $\tilde{\mathcal{B}}_{3l}$ is very tricky (Proposition \ref{prop:tricky}). Such a parameterization has not been seen before. Secondly, the phase shift of a longer soliton may become negative after collision (Theorem \ref{th:main}). This phenomenon is observed only in the $D_4^{(3)}$-SCA (\cite{Yd}) and ours.
\section{Preliminaries}
Let $\mathfrak{g}=G_2^{(1)}$ be the affine algebra whose Dynkin diagram is depicted as follows.
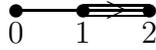
\begin{center}
\begin{picture}(50,10)
	
	\put(-5,5){\circle*{5}}
	\put(20,5){\circle*{5}}
	\put(45,5){\circle*{5}}

	\put(-5,5){\line(1,0){25}}
	\multiput(20,3)(0,2){3}{\line(1,0){25}}
	
	
	\put(-8,-7){0}
	\put(17,-7){1}
	\put(42,-7){2}

	\put(28,2.25){$>$}
\end{picture}

\end{center}
Let $\{\alpha_0,\alpha_1, \alpha_2\}$ be the set of simple roots, $\{h_0,h_1,h_2\}$ be the set of simple coroots, and $\{\Lambda_0,\Lambda_1, \Lambda_2\}$ be the set of fundamental weights of $\mathfrak{g}$.  The null root $\delta$ and the canonical central element $c$ are given by:
$$
\delta=\alpha_0+2\alpha_1+3\alpha_2, \qquad c=h_0+2h_1+h_2.
$$
Let $P=\text{span}_\mathbb{Z}\{\Lambda_0, \Lambda_1, \Lambda_2, \delta\}$ be the \emph{weight lattice}. $U_q(\mathfrak{g})$ denotes the quantum group associated with $\mathfrak{g}.$
\subsection{Crystals}
In this section we give the basic definitions regarding crystals.  Our exposition follows \cite{HK}.  Let $I=\{0,1,2\}$.  
\begin{definition}
A \emph{crystal} associated with $U_q(\mathfrak{g})$ is a set $\mathcal{B}$ together with maps $\text{wt}:\mathcal{B}\to P, \tilde{e}_i, \tilde{f}_i : \mathcal{B}\to \mathcal{B}\cup \{0\},$ and $\varepsilon_i, \varphi_i:\mathcal{B}\to \mathbb{Z}\cup \{-\infty\}$, for $i\in I$ satisfying the following properties:
\begin{enumerate}
\item $\varphi_i(b)=\varepsilon_i(b)+\langle h_i, \text{wt}(b)\rangle$ for all $i \in I,$
\item $\text{wt}(\tilde{e}_ib)=\text{wt}(b)+\alpha_i$ if $\tilde{e}_ib\in \mathcal{B}$,
\item $\text{wt}(\tilde{f}_ib)=\text{wt}(b)-\alpha_i$ if $\tilde{f}_ib\in \mathcal{B}$,
\item $\varepsilon_i(\tilde{e}_ib)=\varepsilon_i(b)-1, \varphi_i(\tilde{e}_ib)=\varphi_i(b)+1$ if $\tilde{e}_ib\in \mathcal{B},$
\item $\varepsilon_i(\tilde{f}_ib)=\varepsilon_i(b)+1, \varphi_i(\tilde{f}_ib)=\varphi_i(b)-1$ if $\tilde{f}_ib\in \mathcal{B},$
\item $\tilde{f}_ib=b'$ if and only if $b=\tilde{e}_ib'$ for $b,b'\in \mathcal{B}, i\in I,$
\item if $\varphi_i(b)=-\infty$ for $b\in\mathcal{B},$ then $\tilde{e}_ib=\tilde{f}_ib=0.$
\end{enumerate}
\end{definition}
A crystal $\mathcal{B}$ can be regarded as a colored, oriented graph by defining
$$
b\stackrel{i}{\to} b' \iff \tilde{f}_ib=b'.
$$
The \emph{tensor product} $\mathcal{B}_1\otimes \mathcal{B}_2$ of crystals $\mathcal{B}_1$ and $\mathcal{B}_2$ is the set $\mathcal{B}_1\times \mathcal{B}_2$ together with the following maps:
\begin{enumerate}
\item $\text{wt}(b_1\otimes b_2)=\text{wt}(b_1)+\text{wt}(b_2),$
\item $\varepsilon_i(b_1\otimes b_2)=\max(\varepsilon_i(b_1),\varepsilon_i(b_2)-\langle h_i,\text{wt}(b_1)\rangle),$
\item $\varphi_i(b_1\otimes b_2)=\max(\varphi_i(b_2),\varphi_i(b_1)+\langle h_i,\text{wt}(b_2)\rangle),$
\item $\tilde{e}_i(b_1\otimes b_2)=\left \{ 
	\begin{array}{l}\tilde{e}_ib_1\otimes b_2, \text{ if } \varphi_i(b_1)\geq \varepsilon(b_2),\\
		b_1\otimes \tilde{e}_ib_2, \text{ if } \varphi_i(b_1)< \varepsilon(b_2),
	\end{array}\right .$
\item $\tilde{f}_i(b_1\otimes b_2)=\left \{ 
	\begin{array}{l}\tilde{f}_ib_1\otimes b_2, \text{ if } \varphi_i(b_1)> \varepsilon(b_2),\\
		b_1\otimes \tilde{f}_ib_2, \text{ if } \varphi_i(b_1)\leq \varepsilon(b_2),
	\end{array}\right .$
\end{enumerate}
where we write $b_1\otimes b_2$ for $(b_1,b_2)\in \mathcal{B}_1\times \mathcal{B}_2$, and understand $b_1\otimes 0=0\otimes b_2=0.$  $\mathcal{B}_1\otimes \mathcal{B}_2$ is a crystal, as can easily be shown.

\begin{definition}
Let $\mathcal{B}_1$ and $\mathcal{B}_2$ be $U_q(\mathfrak{g})$-crystals.  A \emph{crystal isomorphism} is a bijective map $\Psi:\mathcal{B}_1\cup \{0\}\to \mathcal{B}_2\cup \{0\}$ such that
\begin{enumerate}
\item $\Psi(0)=0,$
\item if $b\in \mathcal{B}_1$ and $\Psi(b)\in \mathcal{B}_2,$ then $\text{wt}(\Psi(b))=\text{wt}(b), \varepsilon_i(\Psi(b))=\varepsilon_i(b),\varphi_i(\Psi(b))=\varphi_i(b)$ for all $i\in I,$
\item if $b,b'\in \mathcal{B}_1, \Psi(b), \Psi(b')\in \mathcal{B}_2$ and $\tilde{f}_ib=b',$ then $\tilde{f}_i\Psi(b)=\Psi(b')$ and $\Psi(b)=\tilde{e}_i\Psi(b')$ for all $i\in I.$
\end{enumerate}
\end{definition}
\subsection{Perfect crystal $\mathcal{B}_l$ for $G_2^{(1)}$}

For a positive integer $l$ define the set 
$$
\mathcal{B}_l=\left \{
\begin{array}{c|l} b=(x_1,x_2,x_3,\overline{x}_3,\overline{x}_2,\overline{x}_1) \in (\mathbb{Z}_{\geq 0}/3)^6 & \begin{array}{l}
			3x_3\equiv 3\overline{x}_3\pmod{2}\\
			\sum_{i=1,2} (x_i+\overline{x}_i)+\frac{x_3+\overline{x}_3}{2}\leq l\\
			x_1,\overline{x}_1,x_2-x_3,\overline{x}_3-\overline{x}_2\in \mathbb{Z}
			\end{array} 
\end{array}\right \}.
$$
For $b=(x_1,x_2,x_3,\overline{x}_3,\overline{x}_2,\overline{x}_1)\in \mathcal{B}_l$ we denote
\begin{equation}
s(b)=x_1+x_2+\frac{x_3+\overline{x}_3}{2}+\overline{x}_1, t(b)=x_2+\frac{x_3+\overline{x}_3}{2}
\end{equation}
and define
\begin{equation}
A=\{0,z_1,z_1+z_2,z_1+z_2+3z_4,z_1+z_2+z_3+3z_4,2z_1+z_2+z_3+3z_4\}.
\end{equation}

Now we define conditions $(E_1)$-$(E_6)$ and $(F_1)$-$(F_6)$ as follows.

\begin{equation}
\left \{ 
\begin{array}{ll}
(F_1) & z_1+z_2+z_3+3z_4\leq 0, z_1+z_2+3z_4 \leq 0, z_1+z_2\leq 0, z_1 \leq 0,\\
(F_2) & z_1+z_2+z_3+3z_4\leq 0, z_2+3z_4 \leq 0, z_2\leq 0, z_1 > 0,\\
(F_3) & z_1+z_3+3z_4\leq 0, z_3+3z_4 \leq 0, z_4\leq 0, z_2>0, z_1+z_2 > 0,\\
(F_4) & z_1+z_3+3z_4 > 0, z_2+3z_4 > 0, z_4 > 0, z_3\leq 0, z_1+z_3 \leq 0,\\
(F_5) & z_1+z_2+z_3+3z_4> 0, z_2+3z_4 > 0, z_3 > 0, z_1 \leq 0,\\
(F_6) & z_1+z_2+z_3+3z_4> 0, z_1+z_3+3z_4 > 0, z_1+z_3> 0, z_1 > 0.
\end{array}
\right .
\end{equation}
$(E_i), (1\leq i\leq6)$ is defined  from $(F_i)$ by replacing $>$ (resp. $\leq$) with $\geq$ (resp. $<$).
	Then for $b=(x_1,x_2,x_3,\overline{x}_3,\overline{x}_2,\overline{x}_1)\in \mathcal{B}_l$, we define $\tilde{e}_i(b), \tilde{f}_i(b), \varepsilon_i(b), \varphi_i(b), i=0,1,2$ as follows.  We use the convention: $(a)_+=\max(a,0).$
\begin{align}
\tilde{e}_0(b)&=\left \{
\begin{array}{ll}
(x_1-1, \dots)& \text{if } (E_1),\\
(\dots, x_3-1,\overline{x}_3-1, \dots, \overline{x}_1+1) & \text{if } (E_2),\\
(\dots, x_2-\frac{2}{3},x_3-\frac{2}{3}, \overline{x}_3+\frac{4}{3},\overline{x}_2+\frac{1}{3},\dots) & \text{if }(E_3)\text{ and }z_4=-\frac{1}{3}\\
(\dots, x_2-\frac{1}{3},x_3-\frac{4}{3}, \overline{x}_3+\frac{2}{3},\overline{x}_2+\frac{2}{3}, \dots) & \text{if }(E_3)\text{ and }z_4=-\frac{2}{3}\\
(\dots, x_3-2, \dots,\overline{x}_2+1,\dots) & \text{if }(E_3)\text{ and }z_4\neq-\frac{1}{3},-\frac{2}{3}\\
(\dots, x_2-1, \dots, \overline{x}_3+2, \dots)& \text{if } (E_4) \\
(x_1-1,\dots, x_3+1, \overline{x}_3+1,\dots) & \text{if } (E_5)\\
(\dots, \overline{x}_1+1)& \text{if } (E_6),
\end{array}
\right . \label{e0}\\
\tilde{f}_0(b)&=\left \{
\begin{array}{ll}
(x_1+1, \dots)& \text{if } (F_1),\\
(\dots, x_3+1,\overline{x}_3+1, \dots, \overline{x}_1-1) & \text{if } (F_2),\\
(\dots, x_3+1,\overline{x}_3+1, \dots, \overline{x}_1-1) & \text{if } (F_3),\\
(\dots, x_2+\frac{1}{3},x_3+\frac{4}{3}, \overline{x}_3-\frac{2}{3},\overline{x}_2-\frac{2}{3},\dots) & \text{if }(F_4)\text{ and }z_4=\frac{1}{3}\\
(\dots, x_2+\frac{2}{3},x_3+\frac{2}{3}, \overline{x}_3-\frac{4}{3},\overline{x}_2-\frac{1}{3}, \dots) & \text{if }(F_4)\text{ and }z_4=\frac{2}{3}\\
(\dots, x_2+1, \dots,\overline{x}_3-2,\dots) & \text{if }(F_4)\text{ and }z_4\neq \frac{1}{3},\frac{2}{3}\\
(x_1+1,\dots, x_3-1, \overline{x}_3-1,\dots) & \text{if } (F_5),\\
(\dots, \overline{x}_1-1)& \text{if } (F_6),
\end{array}
\right .\label{f0}
\end{align}
\begin{align}
\tilde{e}_1(b)&=\left \{ \begin{array}{ll}
(\dots, \overline{x}_2+1,\overline{x}_1-1) & \text{if }z_2\geq (-z_3)_+,\\
(\dots, x_3+1, \overline{x}_3-1, \dots) & \text{if }x_2<0\leq z_3,\\
(x_1+1, x_2-1, \dots) & \text{if }(z_2)_+<-z_3,
\end{array}\
\right .\\
\tilde{f}_1(b)&=\left \{ \begin{array}{ll}
(x_1-1,x_2+1,\dots) & \text{if }(z_2)_+\leq -z_3,\\
(\dots, x_3-1, \overline{x}_3+1, \dots) & \text{if }x_2\leq 0 < z_3,\\
(\dots, \overline{x}_2-1, \overline{x}_1+1) & \text{if }z_2>(-z_3)_+,
\end{array}
\right .\\
\tilde{e}_2(b)&=\left \{\begin{array}{ll}
(\dots, \overline{x}_3+\frac{2}{3}, \overline{x}_2-\frac{1}{3},\dots) & \text{if } z_4\geq 0,\\
(\dots, x_2+\frac{1}{3},x_3-\frac{2}{3},\dots) & \text{if }z_4<0,
\end{array}
\right .\\
\tilde{f}_2(b)&=\left \{\begin{array}{ll}
(\dots, x_2-\frac{1}{3}, x_3+\frac{2}{3}) & \text{if } z_4\leq 0,\\
(\dots, \overline{x}_3-\frac{2}{3},\overline{x}_2+\frac{1}{3},\dots) & \text{if }z_4>0.
\end{array}
\right .
\end{align}
$$
\begin{array}{l}
\varepsilon_1(b)=\overline{x}_1+(\overline{x}_3-\overline{x}_2+(x_2-x_3)_+)_+,\quad \varphi_1(b)=x_1+(x_3-x_2+(\overline{x}_2-\overline{x}_3)_+)_+,\\
\varepsilon_2(b)=3\overline{x}_2+\frac{3}{2}(x_3-\overline{x}_3)_+,\quad \varphi_2(b)=3x_2+\frac{3}{2}(\overline{x}_3-x_3)_+,\\
\varepsilon_0(b)=l-s(b)+\max \:A-(2z_1+z_2+z_3+3z_4), \quad \varphi_0(b)=l-s(b)+\max \: A.
\end{array}
$$

For $b\in \mathcal{B}_l$ if $b'=\tilde{e}_i(b)$ or $b'=\tilde{f}_i(b)$ does not belong to $\mathcal{B}_l,$ namely, if $x_j$ or $\overline{x}_j$ of $b'$ for some $j$ is negative or $s(b')$ exceeds $l$, we understand $b'$ to be 0.
\begin{theorem}[\cite{MMO}]
For the quantum affine algebra $U_q'(G_2^{(1)})$ the set $\mathcal{B}_l$ with the maps $\tilde{e_i}, \tilde{f_i}, \varepsilon_i, \varphi_i, i=0,1,2$ is a perfect crystal of level $l$.  
\end{theorem}
\begin{figure}[h]
\caption{The crystal graph of $\mathcal{B}_1$.  Each $\swarrow$ is a 1-arrow, each $\searrow$ is a 2-arrow, and arrows pointing up are 0-arrows.}
\begin{tikzpicture}[scale=1]
\draw[->](0,1)  node[anchor=south west] {1 } -- (-.4,.5)
			node[anchor=north east]  { 2 };
\draw[->](-.4,.15) -- (0,-.25) node[anchor=north west] {  $2_1$ };
\draw[->](.55,-.80) -- (.75,-1) node[anchor=north west]  {$2_2$ };
\draw[->](1.2,-1.5)-- (1.5,-1.75) node[anchor=north west] { 3 };
\draw[->](.8,-1.35)--(.45,-1.75) node[anchor=north east] {$2_3$ };
\draw[->](1.5,-2.1)--(1.025,-2.5) node[anchor=north] {$ 0 $};
\draw[->](.2,-2.2)--(.4,-2.4) node[anchor=north] {$ \hat{0} $};
\draw[->](.5,-2.95)--(1.5,-3.25) node[anchor=north west] {$ \overline{2}_3 $};
\draw[->](1.025,-2.95)--(.3,-3.25) node[anchor=north east] {$  \overline{3} $};
\draw[->](.3,-3.6)--(.75,-4) node[anchor=north west] {$\overline{2}_2$};
\draw[->](1.5,-3.6)--(1.15,-4);
\draw[->](1.3,-4.55)--(1.5,-4.75) node[anchor=north west] {$ \overline{2}_1 $};
\draw[->](2,-5.3)--(2.25,-5.5) node[anchor=north west] {$ \overline{2} $};
\draw[->](2.25,-5.85)--(1.90,-6.25) node[anchor=north east] {$ \overline{1} $};
\draw[->] (1.5,-6.25) -- (-3,-2.85) node[anchor=south east] {$\varnothing$};
\draw[->] (-3,-2.4) ..controls(-2,-.225) ..(0,1.15);
\draw[->](2.4,-5.5)--(1.75,-2.2);
\draw[->](1.7,-4.8)--(1.05,-1.5);
\draw[->](1,-4.1)--(.35,-.8);
\draw[->](.1,-3.3)--(-.55,0);
\end{tikzpicture}
\end{figure}
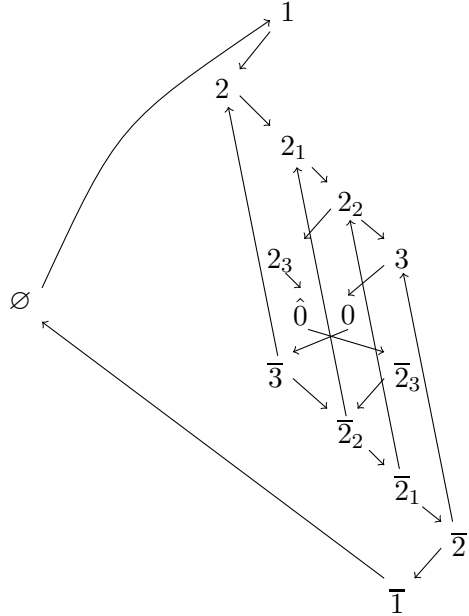
\noindent\emph{Remark:} The $U_q'(D_4^{3})$-crystal $\mathcal{B}^{1,l}$ is contained in $\mathcal{B}_l$.

Let $$B=\{1,2,2_1,2_2,2_3,3,0,\hat{0},\overline{3},\overline{2}_3,\overline{2}_2,\overline{2}_1,\overline{2},\overline{1}\},$$ and define the following partial ordering on $B$:
$$
1 < 2 < 2_1 < 2_2 <\begin{array}{c c c c c }2_3 & < & \hat{0} & < & \overline{2}_3\\
						3 & < & 0 & < &\overline{3}  \end{array}
< \overline{2}_2 < \overline{2}_1 < \overline{2} < \overline{1}.$$
\begin{proposition}[\cite{Yn} 3.9, see also \cite{KM}]
The $U_q(G_2)$ crystal $\mathcal{B}(j\Lambda_1)$ is given by the following set of restricted semistandard Young tableaux whose entries are in $B$:
$$
\left \{ \begin{array}{c|l}\begin{array}{l} \begin{array}{|c|c|c|c|}
		\hline
		\alpha_1& \alpha_2 & \cdots & \alpha_j\\
		\hline
	\end{array} \\
	=\alpha_j\otimes \alpha_{j-1}\otimes \cdots \otimes \alpha_1
\end{array}
&
\begin{array}{l}
				\alpha_k \leq \alpha_{k+1} (k=1,2,\dots, j-1)\\
				t_{2_3}(b)+t_{0}(b)+t_{\hat{0}}(b)+\overline{t}_{2_3}(b)\leq 1\\
				t_{2_1}(b)+t_{2_2}(b)+t_{2_3}(b)\leq 1\\
				\overline{t}_{2_3}(b)+\overline{t}_{2_2}(b)+\overline{t}_{2_1}(b)\leq 1\\
				t_{2_3}(b)+\textup{sgn}(t_3(b))+t_{\hat{0}}(b)\leq 1\\
				t_{\hat{0}}(b)+\textup{sgn}(\overline{t}_3(b))+\overline{t}_{2_3}(b)\leq 1
			\end{array}\end{array}\right \}
$$
where $t_i=\bigl |\{\alpha_k=i|k=1,2,\dots, j\}\bigr |$, $\overline{t}_i=\bigl |\{\alpha_k=\overline{i}|k=1,2,\dots, j\}\bigr |$, and $$
\textup{sgn}(n)=\left \{ \begin{array}{l l}1 & \text{if }n>0,\\
	0& \text{if }n=0. \end{array} \right . 
$$
\end{proposition}
We let $\varnothing$ denote the empty tableau, and do not put a box around a tableau with a single entry.
\begin{proposition}[\cite{Yn} prop. 2.1]
Forgetting the 0-arrows, $\mathcal{B}_l\cong \bigoplus_{j=0}^l\mathcal{B}(j\Lambda_1)$ as a $U_q(G_2)$-crystal.	
\end{proposition}
The tableau $\begin{array}{|c|c|c|c|}
		\hline
		\alpha_1& \alpha_2 & \cdots & \alpha_j\\
		\hline
	\end{array}$
has the coordinate following representation:
\begin{align*}
\bigg (t_1,\frac{3t_2+2t_{2_1}+t_{2_2}+t_{2_3}+t_{\hat{0}}}{3},\frac{2t_{2_1}+4t_{2_2}+t_{2_3}+6t_3+t_{\hat{0}}+3t_{0}+3\overline{t}_{2_3}}{3}, \qquad \qquad\\  \frac{3t_{2_3}+t_{\hat{0}}+3t_{0}+6\overline{t}_3+\overline{t}_{2_3}+4\overline{t}_{2_2}+2\overline{t}_{2_1}}{3},
\frac{t_{\hat{0}}+\overline{t}_{2_3}+\overline{t}_{2_2}+2\overline{t}_{2_1}+3\overline{t}_2}{3} ,\overline{t}_1\bigg ), 
\end{align*}
where $t_k,\overline{t}_k,t_{0},$ and $t_{\hat{0}}$ are as defined above.

\subsection{Crystal $\mathcal{B}_\natural$ for $G_2^{(1)}$}
In this section we define the $U_q'(G_2^{(1)})$-crystal $\mathcal{B}_\natural$ whose crystal graph is given below:
\begin{center}
\begin{tikzpicture}[scale=1]
\draw[->](0,0) node[anchor=east]{$1'$} --node[pos=0.5,anchor=south]{2}  (1,0) node[anchor=west]{$2'$};
\draw[->](1.45,0) -- node[pos=0.5,anchor=south]{1}(2.45,0) node[anchor=west]{$3'$};
\draw[->](2.9,0) -- node[pos=0.5,anchor=south]{2}(3.9,0) node[anchor=west]{$0'$};
\draw[->](4.35,0) -- node[pos=0.5,anchor=south]{2}(5.35,0) node[anchor=west]{$\overline{3}'$};
\draw[->](5.8,0) -- node[pos=0.5,anchor=south]{1}(6.8,0) node[anchor=west]{$\overline{2}'$};
\draw[->](7.25,0) -- node[pos=0.5,anchor=south]{2}(8.25,0) node[anchor=west]{$\overline{1}'$};
\draw[->](6.8,.25) ..controls(4.8,1.25) and (2,1.25)   ..(0,.25) ;
\draw[->](8.25,-.25)..controls(6.25,-1.25) and (3.45,-1.25) ..(1.45,-.25);
\node at (3.4,1.25) {0};
\node at (4.85,-.75) {0};
\end{tikzpicture}
\end{center}
It is the $U_q'(G_2^{(1)})$-crystal $\mathcal{B}^{2,1}$ corresponding to the vertex 2 of the Dynkin diagram.
\subsection{Combinatorial $R:\mathcal{B}_l\otimes \mathcal{B}_{l'}\to \mathcal{B}_{l'}\otimes \mathcal{B}_l$}

In this section, we define the combinatorial $R$ matrix for $\mathcal{B}_l\otimes \mathcal{B}_{l'}$ and the energy function $H:\mathcal{B}_l\otimes \mathcal{B}_{l'}\to \mathbb{Z}$.

\begin{proposition}[\cite{KMN}]
There exist a unique crystal isomorphism $\mathcal{R}:\mathcal{B}_l\otimes \mathcal{B}_{l'}\to \mathcal{B}_{l'}\otimes \mathcal{B}_l$ and a function $H:\mathcal{B}_l\otimes \mathcal{B}_{l'}\to \mathbb{Z}$ unique up to an additive constant satisfying the following property: for any $b\in \mathcal{B}_l$, $b'\in \mathcal{B}_{l'},$ and $i\in I$ such that $\tilde{e_i}(b\otimes b')\neq 0$, 
\begin{equation}
H(\tilde{e}_i(b\otimes b'))=\left \{ \begin{array}{ll}
		H(b\otimes b')+1 & \text{if }i=0, \varphi_0(b)\geq \varepsilon_0(b'), \varphi_0(\tilde{b}')\geq \varepsilon_0(\tilde{b}),\\
		H(b\otimes b')-1 & \text{if }i=0, \varphi_0(b) < \varepsilon_0(b'), \varphi_0(\tilde{b}') < \varepsilon_0(\tilde{b}),\\
		H(b\otimes b') & \text{otherwise,}
\end{array}
\right .
\end{equation}
where $\tilde{b}'\otimes \tilde{b}=\mathcal{R}(b\otimes b').$  $H$ is called an \emph{energy function} on $\mathcal{B}_l\otimes \mathcal{B}_{l'}.$ 
\end{proposition}
\begin{proposition}[Yang-Baxter Equation]
The following equation holds on $\mathcal{B}_l\otimes \mathcal{B}_{l'}\otimes \mathcal{B}_{l''}$
\begin{equation}
(\mathcal{R}\otimes 1)(1\otimes \mathcal{R})(\mathcal{R}\otimes 1)=(1\otimes \mathcal{R})(\mathcal{R}\otimes 1)(1\otimes \mathcal{R}), \label{YBE}
\end{equation}
where $1$ denotes the identity map.
\end{proposition}
We  define $\mathcal{R}_{i,i+1}$ to be the map $1^{\otimes i-1}\otimes \mathcal{R}\otimes 1^{\otimes L-i-1}:\bigotimes_{i=1}^L \mathcal{B}_{l_i}\to \bigotimes_{i=1}^L \mathcal{B}_{l_i}$.  In this notation, \eqref{YBE} becomes:
\begin{equation}
\mathcal{R}_{i,i+1}\mathcal{R}_{i+1,i+2}\mathcal{R}_{i,i+1}=\mathcal{R}_{i+1,i+2}\mathcal{R}_{i,i+1}\mathcal{R}_{i+1,i+2}
\end{equation}
for $L\geq 3, i=1,2,\dots, L-2.$

\begin{definition}[Affinization]
The \emph{affinization} of $\mathcal{B}_l$ is defined to be the set
$$
\{ z^nb | n\in \mathbb{Z}, b\in \mathcal{B}_l\}
$$
where the action of $\tilde{e}_i, \tilde{f}_i$ is defined as:
\begin{eqnarray*}
\tilde{e}_i(z^nb)&=&z^{n+\delta_{i0}}\tilde{e}_i(b)\\
\tilde{f}_i(z^nb)&=&z^{n-\delta_{i0}}\tilde{f}_i(b).
\end{eqnarray*}
\end{definition}
The combinatorial $R$-matrix $\mathcal{R}^{\text{Aff}}:\text{Aff}(\mathcal{B}_l)\otimes \text{Aff}(\mathcal{B}_{l'})\to\text{Aff}(\mathcal{B}_{l'})\otimes \text{Aff}(\mathcal{B}_{l})$ is given by:
\begin{equation}
\mathcal{R}^{\text{Aff}}(z^m b\otimes z^n b')=z^{n+H(b\otimes b')}\tilde{b}' \otimes z^{m-H(b\otimes b')}\tilde{b}
\end{equation}
where $\tilde{b}'\otimes \tilde{b}=\mathcal{R}(b\otimes b')$.

\subsection{Combinatorial $R$ matrix for $A_1^{(1)}$}
Define $\tilde{\mathcal{B}}_l$ to be the set  $\{(x_1,x_2)\in \mathbb{Z}_{\geq 0}^2 | x_1+x_2=l\}$ together with the maps:
$$
\tilde{e}_0^A(x_1,x_2)=(x_1-1,x_2+1), \qquad \tilde{e}_1^A(x_1,x_2)=(x_1+1, x_2-1),
$$
$$
\tilde{f}_0^A(x_1,x_2)=(x_1+1,x_2-1), \qquad \tilde{f}_1^A(x_1,x_2)=(x_1-1, x_2+1),
$$
$$
\varepsilon_i^A(x_1,x_2)=x_{i+1}, \qquad \varphi_i^A(x_1,x_2)=x_{2-i},\qquad i=0,1.
$$
For $b\in \mathcal{\tilde{B}}_l$ if $b'=\tilde{e}^A_i(b)$ or $b'=\tilde{f}^A_i(b)$ and some $x_i$ is negative in $b'$, then we understand $b'$ to be 0.  This defines a perfect $U_q'(A_1^{(1)})$-crystal.

The crystal isomorphism $\hat{\mathcal{R}}:\tilde{\mathcal{B}}_l\otimes \tilde{\mathcal{B}}_{l'}\to \tilde{\mathcal{B}}_{l'}\otimes \tilde{\mathcal{B}}_{l}$ is given by 
$$
\hat{\mathcal{R}}((x_1,x_2)\otimes (y_1,y_2))=(y_1',y_2')\otimes(x_1',x_2'),
$$
where
\begin{eqnarray}
x_1'&=&x_1+\min(y_{1},x_{2})-\min(y_{2},x_{1}),\\
x_2'&=&x_2+\min(y_{2},x_{1})-\min(y_{1},x_{2}),\\
y_1'&=&y_1+\min(y_{2},x_1)-\min(y_1,x_{2}),\\
y_2'&=&y_2+\min(y_{1},x_2)-\min(y_2,x_{1}).
\end{eqnarray}
To $(x_1,x_2)\in \tilde{\mathcal{B}}_l$ we associate the following tableau: 
$$\begin{array}{l @{} l}\underbrace{\begin{array}{|c|c|c|c|}\hline 1 & 1 & \cdots & 1 \\ \hline\end{array}}_{x_1}& 
\underbrace{\begin{array}{c|c|c|c|}\hline  2 & 2 & \cdots & 2\\ \hline \end{array}}_{x_2}\end{array}.$$ 

An energy function $\hat{H}:\tilde{\mathcal{B}}_l\otimes \tilde{\mathcal{B}}_{l'}\to \mathbb{Z}$ is given by:
\begin{equation}
\hat{H}((x_1,x_2)\otimes (y_1,y_2))=-\min(y_2,x_1).
\end{equation}
The combinatorial $R$-matrix $\hat{\mathcal{R}}^{\text{Aff}}:\text{Aff}(\tilde{\mathcal{B}}_l)\otimes \text{Aff}(\tilde{\mathcal{B}}_{l'})\to\text{Aff}(\tilde{\mathcal{B}}_{l'})\otimes \text{Aff}(\tilde{\mathcal{B}}_{l})$ for $A_1^{(1)}$ is given by:
\begin{equation}
\hat{\mathcal{R}}^{\text{Aff}}(z^m b\otimes z^n b')=z^{n+\hat{H}(b\otimes b')}\tilde{b}' \otimes z^{m-\hat{H}(b\otimes b')}\tilde{b}
\end{equation}
where $\tilde{b}'\otimes \tilde{b}=\hat{\mathcal{R}}(b\otimes b')$.

\section{Combinatorial $R$ matrix for $G_2^{(1)}$}
In this section, we give the combinatorial $R$ matrix for the $U_q'(G_2^{(1)})$-crystals $\mathcal{B}_l\otimes \mathcal{B}_1$ and $\mathcal{B}_\natural \otimes \mathcal{B}_1$.  This is done using  an algorithm similar to the column insertion scheme in \cite{L}.

\subsection{Combinatorial $R$ matrix for highest weight elements}
First we find the highest weight elements of $\mathcal{B}_l\otimes \mathcal{B}_1$ using the following lemma.
\begin{lemma}[\cite{Ka}]
An element $b_1 \otimes b_2 \in \mathcal{B}(\lambda)\otimes \mathcal{B}(\mu)$ is highest weight if and only if 
$\tilde{e}_ib_1=0 $ and $\tilde{e}_i^{\lambda(h_i)+1}b_2=0$ for $i = 1,2$.  
\end{lemma}
By proposition 2.2, as a $U_q(G_2)$ crystal,
\begin{eqnarray*}
\mathcal{B}_l\otimes \mathcal{B}_1&\cong& \bigoplus_{j=0}^l\mathcal{B}(j\Lambda_1)\otimes \bigoplus_{k=0}^1 \mathcal{B}(k\Lambda_1)\\
	&=&\bigoplus_{j=0}^l(\mathcal{B}(j\Lambda_1)\otimes\mathcal{B}(0))\oplus \bigoplus_{j=0}^l(\mathcal{B}(j\Lambda_1) \otimes \mathcal{B}(\Lambda_1)).
\end{eqnarray*}
Therefore, the highest weight elements are those of the form $(j,0,0,0,0,0)\otimes b'$ where $b'\in \mathcal{B}(0)$ or $\mathcal{B}(\Lambda_1)$ and satisfies $\tilde{e}_i^{j\delta_{i1}+1}(b')=0$.  We therefore obtain by inspection of the crystal graph in figure 1 that the following is a complete set of highest weight elements of $\mathcal{B}_l\otimes \mathcal{B}_1$:

$$
\begin{array}{|cc|c|}
\hline
 && \text{Classical weight}\\
\text{Highest weight element } b_1\otimes b_2&& \text{wt}(b_1\otimes b_2)\\
\hline \hline
(j,0,0,0,0,0)\otimes (1,0,0,0,0,0) & 0\leq j \leq l&(j+1)\Lambda_1\\
(j,0,0,0,0,0)\otimes (0,0,0,0,0,0) & 0\leq j \leq l&j\Lambda_1\\
(j,0,0,0,0,0)\otimes (0,1,0,0,0,0) & 1\leq j \leq l&(j-1)\Lambda_1+3\Lambda_2\\
(j,0,0,0,0,0)\otimes (0,\frac{1}{3},\frac{1}{3},1,0,0) & 1\leq j \leq l&(j-1)\Lambda_1+2\Lambda_2\\
(j,0,0,0,0,0)\otimes (0,0,1,1,0,0) & 1\leq j \leq l&j\Lambda_1\\
(j,0,0,0,0,0)\otimes (0,0,0,2,0,0) & 2\leq j \leq l&(j-2)\Lambda_1+3\Lambda_2\\
(j,0,0,0,0,0)\otimes (0,0,0,0,0,1) & 1\leq j \leq l&(j-1)\Lambda_1\\
\hline
\end{array}
$$
\begin{proposition}
Let $\mathcal{R} : \mathcal{B}_l \otimes \mathcal{B}_1 \to \mathcal{B}_1\otimes \mathcal{B}_l$ be the $U_q'(G_2^{(1)})$  crystal isomorphism.  Then we have, letting $b_1=(j,0,0,0,0,0)$:

\begin{equation}
\mathcal{R}(b_1\otimes b_2)=\left \{ \begin{array}{ll}
	(1,0,0,0,0,0)\otimes (l,0,0,0,0,0) & \text{if }b_2=(1,0,0,0,0,0), \\ &\text{and } j=l,\\
	(0,0,0,0,0,0)\otimes (l,0,0,0,0,0) & \text{if }b_2=(1,0,0,0,0,0) ,\\ &\text{and } j=l-1,\\
	(1,0,0,0,0,0)\otimes (j+1,0,0,0,0,1) & \text{if }b_2=(1,0,0,0,0,0), \\ &\text{and } 0\leq j \leq l-2,\\
	(1,0,0,0,0,0)\otimes (l-1,0,0,0,0,0) & \text{if }b_2=(0,0,0,0,0,0), \\ &\text{and }j=l,\\
	(0,0,0,0,0,0)\otimes (j,0,0,0,0,0) & \text{if }b_2=(0,0,0,0,0,0), \\ &\text{and } 0 \leq j \leq l-1,\\
	(1,0,0,0,0,0)\otimes (l-1,1,0,0,0,0) & \text{if }b_2=(0,1,0,0,0,0), \\ &\text{and } j=l,\\
	(1,0,0,0,0,0)\otimes (j-1,1,1,1,0,0) & \text{if }b_2=(0,1,0,0,0,0), \\ &\text{and } 1\leq j \leq l-1,\\
	(1,0,0,0,0,0)\otimes (j-1,\frac{1}{3},\frac{1}{3},1,0,0) & \text{if }b_2=(0,\frac{1}{3},\frac{1}{3},1,0,0), \\ &\text{and }1 \leq j \leq l,\\
	(1,0,0,0,0,0)\otimes (j-1,0,1,1,0,0) & \text{if }b_2=(0,0,1,1,0,0), \\ &\text{and } 1 \leq j \leq l, \\
	(1,0,0,0,0,0)\otimes (j-2,1,0,0,0,0) & \text{if }b_2=(0,0,0,2,0,0), \\ &\text{and } 2\leq j \leq l ,\\
	(1,0,0,0,0,0)\otimes (0,0,0,0,0,1) & \text{if }b_2=(0,0,0,0,0,1), \\ &\text{and } j=1, \\
	(1,0,0,0,0,0)\otimes (j-2,0,0,0,0,0) & \text{if }b_2=(0,0,0,0,0,1), \\ &\text{and } 2\leq j \leq l . \end{array}\right .
\end{equation}
\end{proposition}

\begin{proof}
Recall from the remark following theorem 1 that the $D_4^{(3)}$-crystal $\mathcal{B}^{1,l}$ is contained in the $G_2^{(1)}$-crystal $\mathcal{B}_l$.  So when $(j,0,0,0,0,0)\otimes b'$ is a $D_4^{(3)}$ highest weight element, the result follows from \cite{Yd} proposition 3.7, which leaves us to check
$$\mathcal{R}((j,0,0,0,0, 0)\otimes  (0,\frac{1}{3},\frac{1}{3},1,0,0))=(1,0,0,0,0,0)\otimes (j-1,\frac{1}{3},\frac{1}{3},1,0,0), 1\leq j \leq l.\\$$
There are two cases to consider.\\
\emph{Case i:} $l=1$.  In this case
\begin{align*}
\mathcal{R}((1,0,\dots,0)\otimes(0,\frac{1}{3},\frac{1}{3},1,0,0))&=\mathcal{R}(\tilde{f}_0\tilde{f}_1\tilde{f}_2^2\tilde{f}_1((0,0,0,0,0,0)\otimes(1,0,0,0,0,0))),\\
	&=\tilde{f}_0\tilde{f}_1\tilde{f}_2^2\tilde{f}_1\mathcal{R}((0,0,0,0,0,0)\otimes(1,0,0,0,0,0)),\\
	&=\tilde{f}_0\tilde{f}_1\tilde{f}_2^2\tilde{f}_1(0,0,0,0,0,0)\otimes(1,0,0,0,0,0),\\
	&=(1,0,0,0,0,0)\otimes(0,\frac{1}{3},\frac{1}{3},1,0,0).\\
\end{align*}
\emph{Case ii:} $l>1$.  In this case
\begin{align*}
\mathcal{R}((j,0,\dots,0)\otimes(0,\frac{1}{3},\frac{1}{3},1,0,0))&=\mathcal{R}(\tilde{f}_0^j\tilde{f}_1\tilde{f}_2^2\tilde{f}_1((0,0,0,0,0,0)\otimes(1,0,0,0,0,0))),\\
	&=\tilde{f}_0^j\tilde{f}_1\tilde{f}_2^2\tilde{f}_1\mathcal{R}((0,0,0,0,0,0)\otimes(1,0,0,0,0,0)),\\
	&=\tilde{f}_0^j\tilde{f}_1\tilde{f}_2^2\tilde{f}_1(1,0,0,0,0,0)\otimes(1,0,0,0,0,1),\\
	&=(1,0,0,0,0,0)\otimes(j-1,\frac{1}{3},\frac{1}{3},1,0,0).\\
\end{align*}
\end{proof}

\subsection{Energy function for $G_2^{(1)}$-crystal $\mathcal{B}_l\otimes \mathcal{B}_1$}
\begin{theorem}
The energy function $H:\mathcal{B}_l\otimes \mathcal{B}_1\to \mathbb{Z}, l\geq 1$ is given on the highest weight vectors by the following:
\begin{equation}
H((j,0,\dots,0)\otimes b')=\left \{ \begin{array}{ccl}
			0, &\text{if}& j=l, b'=(1,0,0,0,0,0),\\
			-1, &\text{if}&  j=l-1 \text{ and } b'=(1,0,0,0,0,0),\\
			&\text{or}&    j=l \text{ and } b'=(0,1,0,0,0,0) \text{ or } (0,0,0,0,0,0),\\
			-2 &&\text{otherwise.}
			\end{array} \right . \label{eq1}
\end{equation}
\end{theorem}
\begin{proof} Recall from the remark following theorem 1 that the $D_4^{(3)}$-crystal $\mathcal{B}^{1,l}$ is contained in the $G_2^{(1)}$-crystal $\mathcal{B}_l$.  So when $(j,0,0,0,0,0)\otimes b'$ is a $D_4^{(3)}$ highest weight element, we can use the result of \cite{Yd} corollary 3.8.  This leaves us to verify that $H((j,0,0,0,0,0)\otimes (0,\frac{1}{3},\frac{1}{3},1,0,0))=-2, 1\leq j \leq l$.  There are two cases to consider.\\
\emph{Case i:} $l=1$.  In this case $H((0,0,0,0,0,0)\otimes (1,0,0,0,0,0))=0,$ since $j=0.$  Thus,
\begin{align*}
H((0,0,0,0,0,0)\otimes(0,\frac{1}{3},\frac{1}{3},1,0,0))&=H(\tilde{f}_1\tilde{f}_2^2\tilde{f}_1((0,0,0,0,0,0)\otimes(1,0,0,0,0,0))),\\
	&=H((0,0,0,0,0,0)\otimes(1,0,0,0,0,0)),\\
	&=-1.\\
\intertext{Therefore,}
H((1,0,0,0,0,0)\otimes (0,\frac{1}{3},\frac{1}{3},1,0,0))&=H(\tilde{f}_0((0,0,0,0,0,0)\otimes (0,\frac{1}{3},\frac{1}{3},1,0,0)))\\
		&=-2,
\end{align*}
since $\tilde{f}_0$ acts on the first component.

\noindent \emph{Case ii:} $l>1$.  Then, $H((0,0,0,0,0,0)\otimes(1,0,0,0,0,0))=-2.$  Therefore
\begin{eqnarray*}
H((0,0,0,0,0,0)\otimes(0,\frac{1}{3},\frac{1}{3},1,0,0))&=&H(\tilde{f}_1\tilde{f}_2^2\tilde{f}_1((0,0,0,0,0,0)\otimes(1,0,0,0,0,0))),\\
	&=&H((0,0,0,0,0,0)\otimes(1,0,0,0,0,0)),\\
	&=&-2.
\end{eqnarray*}
By proposition 5 we have $$\mathcal{R}((0,0,0,0,0,0)\otimes (1,0,0,0,0,0))=(1,0,0,0,0,0)\otimes(1,0,0,0,0,1),$$ so by acting on $(1,0,0,0,0,0)\otimes(1,0,0,0,0,1)$ by the appropriate sequence of $\tilde{f}_i$ we see $\mathcal{R}((0,0,0,0,0,0)\otimes (0,\frac{1}{3},\frac{1}{3},1,0,0))=(1,0,0,0,0,0)\otimes (0,\frac{1}{3},\frac{1}{3},1,0,1).$  Now, letting $\tilde{f}_0$ act on $(0,0,0,0,0,0)\otimes (0,\frac{1}{3},\frac{1}{3},1,0,0)$ $j$ times, we get $(j,0,0,0,0,0)\otimes (0,\frac{1}{3},\frac{1}{3},1,0,0)$, where each application of $\tilde{f}_0$ acts on the first component.  However, for the corresponding action on $(1,0,0,0,0,0)\otimes (0,\frac{1}{3},\frac{1}{3},1,0,0)$, we see that each application of $\tilde{f}_0$ acts on the \emph{second} component.  Therefore, 
\begin{eqnarray*}
H((j,0,0,0,0,0)\otimes(0,\frac{1}{3},\frac{1}{3},1,0,0))&=&H((0,0,0,0,0,0)\otimes (0,\frac{1}{3},\frac{1}{3},1,0,0)),\\
	&=&-2.
\end{eqnarray*}
\end{proof}

\subsection{Combinatorial $R$ matrix for $\mathcal{B}_l\otimes \mathcal{B}_1$}

Let $b_1 \otimes b_2 \otimes b_3 \otimes \cdots\otimes b_l$ be a highest weight element of the $G_2$-crystal $\mathcal{B}(\Lambda_1)^{\otimes l}$ and $\mathcal{B}(b_1\; b_2 \; \dots \; b_l)$ be the $G_2$-crystal generated by $b_1 \otimes b_2 \otimes b_3 \otimes \cdots \otimes b_l$.  In analogy with the ``column insertion'' algorithm in \cite{L} let $\xi:\mathcal{B}(1\; 0)\to \mathcal{B}(1), \eta:\mathcal{B}(1 \; 1\; 2)\to \mathcal{B}(1 \; 2\; 1),$ and $\theta:\mathcal{B}(1 \; 1\; 2_3)\to \mathcal{B}(1 \; 2_3 \; 1)$ be the unique isomorphisms between the given crystals.  We remark that the $G_2$-crystal considered in \cite{L} is not the same as our $\mathcal{B}(\Lambda_1)^{\otimes l}$, so the analogy is not precise.  Let 
$$\begin{array}{l}
\xi_i:\mathcal{B}(1)^{\otimes i-1}\otimes \mathcal{B}(1\; 0)\otimes \mathcal{B}(1)^{\otimes l-i-1}\to  \mathcal{B}(1)^{\otimes l}\\
\eta_i:\mathcal{B}(1)^{\otimes i-1}\otimes \mathcal{B}(1\; 1 \; 2)\otimes \mathcal{B}(1)^{\otimes l-i-2}\to \mathcal{B}(1)^{\otimes i-1}\otimes \mathcal{B}(1\; 2 \; 1)\otimes \mathcal{B}(1)^{\otimes l-i-2}\\
\theta_i:\mathcal{B}(1)^{\otimes i-1}\otimes \mathcal{B}(1\; 1 \; 2_3)\otimes \mathcal{B}(1)^{\otimes l-i-2}\to
\mathcal{B}(1)^{\otimes i-1}\otimes \mathcal{B}(1\; 2_3 \; 1)\otimes \mathcal{B}(1)^{\otimes l-i-2}
\end{array}$$
 be the following crystal isomorphisms:
$$\begin{array}{l}
\xi_i=1^{\otimes i-1}\otimes \xi \otimes 1^{\otimes l-i-1}\\
\eta_i=1^{\otimes i-1}\otimes \eta \otimes 1^{\otimes l-i-2}\\
\theta_i=1^{\otimes i-1}\otimes \theta \otimes 1^{\otimes l-i-2}\\
\end{array}$$

\begin{lemma}
Let $\begin{array}{|c|c|c|c|}
			\hline
			\alpha_1 & \alpha_2 & \cdots & \alpha_n\\
			\hline
			\end{array}$
be a semistandard $G_2$ tableau, and let $ \beta \in \mathcal{B}(1)$ be such that $\alpha_1 \otimes \beta\in \mathcal{B}(1\;2)$ (resp.  $\mathcal{B}(1\;2_3)$).  Then the image $\eta_1 \eta_2 \dots \eta_{n-1}(\alpha_n \otimes \alpha_{n-1} \otimes \cdots \otimes \alpha_1 \otimes \beta)$ (resp. $\theta_1 \theta_2 \dots \theta_{n-1}(\alpha_n \otimes \alpha_{n-1} \otimes \cdots \otimes \alpha_1 \otimes \beta)$) is well-defined.
\end{lemma}
\begin{proof}
Since  $\begin{array}{|c|c|c|c|}
			\hline
			\alpha_1 & \alpha_2 & \cdots & \alpha_n\\
			\hline
			\end{array}$
is semistandard we can apply a suitable sequence of $\tilde{e}_i$s so that $\alpha_n \otimes \alpha_{n-1} \otimes \cdots \otimes \alpha_1 \otimes \beta$ becomes $1\otimes 1\otimes \cdots \otimes 1\otimes \beta'$ with $1\otimes \beta'\in \mathcal{B}(1 \; 2)$ (resp. $\mathcal{B}(1 \; 2_3)$).  We can then apply a suitable sequence of $\tilde{e}_i$s to get $1\otimes 1\otimes \cdots \otimes 1 \otimes 2$, (resp. $1\otimes 1\otimes \cdots \otimes 1 \otimes 2_3$).  Now, apply  $\eta_1 \eta_2 \dots \eta_{n-1}$ (resp. $\theta_1 \theta_2 \dots \theta_{n-1}$) on  $1\otimes 1\otimes \cdots \otimes 1 \otimes 2$, (resp. $1\otimes 1\otimes \cdots \otimes 1 \otimes 2_3$).  Use the corresponding sequence of $\tilde{f}_i$s and commute with  $\eta_1 \eta_2 \dots \eta_{n-1}$ (resp. $\theta_1 \theta_2 \dots \theta_{n-1}$) to get the image of $\alpha_n \otimes \alpha_{n-1} \otimes \cdots \otimes \alpha_1 \otimes \alpha_1 \otimes \beta$.
\end{proof}
Now we are ready to prove the main result of this section.
\begin{theorem}
Let $b_1=\begin{array}{|c|c|c|c|}
			\hline
			\alpha_1 & \alpha_2 & \cdots & \alpha_j\\
			\hline
			\end{array}=\alpha_j\otimes \alpha_{j-1}\otimes \cdots \otimes \alpha_1 \in \mathcal{B}_l$
and $b_2=\beta$ or $\varnothing \in \mathcal{B}_1$.  Then, the unique  $U_q'(G_2^{(1)})$-crystal isomorphism $\mathcal{R}:\mathcal{B}_l\otimes\mathcal{B}_1\to \mathcal{B}_1\otimes \mathcal{B}_l$ is given as follows:

\begin{equation}
\mathcal{R}(b_1\otimes b_2)=\left \{ \begin{array}{ll}
			 \alpha_l  \otimes 
				\begin{array}{|c|c|c|c|c|}
				\hline
				\beta & \alpha_1 & \alpha_2 & \cdots & \alpha_{l-1}\\
				\hline
			\end{array}&
			\text{if }\alpha_1\otimes \beta \in \mathcal{B}(1\;1),  j =  l,\\
			\varnothing \otimes 
				\begin{array}{|c|c|c|c|c|}
				\hline
				\beta & \alpha_1 & \alpha_2 & \cdots & \alpha_{l-1} \\
				\hline
			\end{array} &
			\text{if }\alpha_1\otimes \beta \in \mathcal{B}(1\;1),  j =  l-1,\\
			1 \otimes 
				\begin{array}{|c|c|c|c|c|c|}
				\hline
				\beta & \alpha_1 & \alpha_2 & \cdots & \alpha_j  & \overline{1} \\
				\hline
			\end{array}& \text{if }\alpha_1\otimes \beta \in \mathcal{B}(1\;1),  j <  l-1,\\
			\alpha_l \otimes \begin{array}{|c|c|c|c|} \hline \alpha_1 & \alpha_2 & \cdots & \alpha_{l-1} \\ \hline 			
				\end{array}&
			 \text{if }b_2=\varnothing, j=l,\\
			\varnothing \otimes b_1&
			 \text{if }b_2=\varnothing, j<l,\\
			 q_{l+1}  \otimes 
				\begin{array}{|c|c|c|c|}
				\hline
				q_1 & q_2 & \cdots & q_{l}\\
				\hline
			\end{array} & \text{if }\alpha_1\otimes \beta\in \mathcal{B}(1\;2), j=l,\\
			&\text{where } q_{l+1} \otimes q_{l} \cdots \otimes q_1=\\
			&\qquad\eta_1 \eta_2 \dots \eta_{l-1}(b_1 \otimes b_2),\\ 
			
			 p  \otimes 
				\begin{array}{|c|c|c|c|c|}
				\hline
				q_1 & q_2 & \cdots & q_{j}&r\\
				\hline
			\end{array}&\text{if }\alpha_1\otimes \beta\in \mathcal{B}(1\;2),1< j<l,\\
			&\text{where } p\otimes r\otimes q_{j} \otimes q_{j-1} \cdots \otimes q_1=\\
			&\qquad \xi^{-1}_1\eta_1 \eta_2 \dots \eta_{j-1}(b_1\otimes b_2),\\ 
			
			 q_{j+1} \otimes 
				\begin{array}{|c|c|c|c|}
				\hline
				q_1 & q_2 & \cdots & q_{j}\\
				\hline
			\end{array}&\text{if }\alpha_1\otimes \beta\in \mathcal{B}(1\;2_3), j>1,\\
			&\text{where } q_{j+1} \otimes q_{j} \cdots \otimes q_1=\\
			&\qquad \theta_1 \theta_2 \dots \theta_{j-1}(b_1\otimes b_2),\\ 
		 
			 p \otimes 
				\begin{array}{|c|c|c|c|c|c|}
				\hline
				r & \alpha_2 & \alpha_3 & \cdots & \alpha_{j-1}& q\\
				\hline
			\end{array}&\text{if } \alpha_1\otimes \beta \in \mathcal{B}(1 \; 0),\\
 			&\alpha_2\otimes \xi(\alpha_1 \otimes \beta) \in \mathcal{B}(1\;1), \text{ and }j>1,\\
			&\text{where } p\otimes q =\xi^{-1}(\alpha_j), r=\xi(\alpha_1 \otimes \beta),\\
			
			 p \otimes 
				\begin{array}{|c|c|c|c|}
				\hline
				q_1 & q_2 & \cdots & q_{j-1}\\
				\hline
			\end{array}&\text{if }\alpha_1\otimes \beta\in \mathcal{B}(1\;0),\\
			&\alpha_2\otimes \xi(\alpha_1\otimes \beta)\in \mathcal{B}(1\;2),\text{ and }j>1,\\
			 &\text{where } p\otimes q_{j-1} \otimes q_{j-2} \cdots \otimes q_1=\\
			&\qquad \eta_1 \eta_2 \dots \eta_{j-2} \xi_j(b_1\otimes b_2),\\
		
			\alpha_j  \otimes 
				\begin{array}{|c|c|c|c|}\hline \alpha_2 & \alpha_3 & \cdots & \alpha_{j-1} \\ \hline \end{array}&\text{if }\alpha_1\otimes \beta = 1 \otimes \overline{1},  j >1,\\ 
			b_1\otimes b_2 &\text{in all other cases.}
		\end{array}\right . \label{CombR}
\end{equation}
\vspace{2cm}
\end{theorem}
\begin{proof}
By proposition 5, we have:
\begin{equation*}
\mathcal{R}(\begin{array}{|c|}
	\hline
	1 ^j\\
	\hline
\end{array}\otimes b_2)=
			\left \{ \begin{array}{ll}1\otimes \begin{array}{|c|}
				\hline
				1^l\\
				\hline
			\end{array}& \text{if } \beta=1, j=l,\\
			\varnothing \otimes \begin{array}{|c|}
				\hline
				1^l\\
				\hline
			\end{array}& \text{if }\beta=1,j=l-1,\\
			1\otimes \begin{array}{|c|c|}
				\hline
				1^{j+1}& \overline{1}\\
				\hline
			\end{array}& \text{if }\beta=1, j<l-1,\\
			1\otimes \begin{array}{|c|c|c|c|}
				\hline
				1^{l-1}\\
				\hline
			\end{array}& \text{if }b_2=\varnothing, j=l,\\
			\varnothing \otimes \begin{array}{|c|c|c|c|}
				\hline
				1^{j}\\
				\hline
			\end{array} & \text{if }b_2=\varnothing, j<l,\\
			1 \otimes \begin{array}{|c|c|c|c|}
				\hline
				1^{l-1}&2\\
				\hline
			\end{array} & \text{if }\beta=2, j=l,\\
			1 \otimes \begin{array}{|c|c|c|c|}
				\hline
				1^{j-1}&2&0\\
				\hline
			\end{array} & \text{if }\beta=2,1 < j<l,\\
			1 \otimes \begin{array}{|c|c|c|c|}
				\hline
				1^{j-1}&2_3\\
				\hline
			\end{array} & \text{if }\beta=2_3, j>1,\\
			1 \otimes \begin{array}{|c|c|c|c|}
				\hline
				1^{j-1}&0\\
				\hline
			\end{array} & \text{if }\beta=0, j>1,\\
			1\otimes \begin{array}{|c|c|c|c|}
				\hline
				1^{j-2}& 2\\
				\hline\end{array} & \text{if } \beta=\overline{3}, j>1 \\
			1 \otimes \begin{array}{|c|c|c|c|}
				\hline
				1^{j-2}\\
				\hline
			\end{array} & \text{if }\beta=\overline{1}, j>1,\\
			b_1 \otimes b_2 & \text{in all other cases,}
\end{array}
\right .
\end{equation*}
where the symbol $b^j$ denotes that the element $b$ is repeated $j$ times.  Thus we can easily verify that the statement of this theorem is true for the highest weight elements and all that remains is to show that the right-hand side of \eqref{CombR} is  well-defined for all elements of $\mathcal{B}_l\otimes \mathcal{B}_1$.  Let $b_1=\begin{array}{|c|c|c|c|}
			\hline
			\alpha_1 & \alpha_2 & \cdots & \alpha_j\\
			\hline
			\end{array}\in \mathcal{B}_l$
and $b_2=\beta$ or $\varnothing \in \mathcal{B}_1$ be given.  If $\alpha_1\otimes \beta\in \mathcal{B}(1\;0),$ then  $\xi(\alpha_1\otimes \beta)$ is well-defined by the definition of $\xi$.  If $ \alpha_2\otimes \xi(\alpha_1\otimes \beta)\in \mathcal{B}(1\;2)$, and $j>1$, then $\eta_1 \eta_2 \dots \eta_{j-2} \xi_j(\alpha_j \otimes \alpha_{j-1} \otimes \cdots \otimes \alpha_1 \otimes \beta)=\eta_1 \eta_2 \dots \eta_{j-2}(\alpha_j \otimes \alpha_{j-1} \otimes \cdots \otimes \alpha_2\otimes \xi(\alpha_1 \otimes \beta))$ is well-defined by lemma 2.  If $\alpha_1\otimes \beta \in \mathcal{B}(1 \; 2)$ (resp. $\mathcal{B}(1\; 2_3)$) then $\eta_1 \eta_2 \dots \eta_{j-1}(\alpha_j \otimes \alpha_{j-1} \otimes \cdots \otimes\alpha_1 \otimes \beta)$ (resp. $\theta_1 \theta_2 \dots \theta_{j-1}(\alpha_j \otimes \alpha_{j-1} \otimes \cdots \otimes\alpha_1 \otimes \beta)$ is well-defined by lemma 2.  The remaining cases are clearly well-defined since they do not use any of the maps $\xi,\eta,$ or $\theta.$
\end{proof}
\noindent \emph{Example:} Let $b_1=\begin{array}{|c|c|c|} \hline 2 & 2_3 & \overline{2}_1 \\ \hline \end{array},b_2=0,l=4$.  Then $\alpha_1\otimes \beta=2\otimes 0 \in \mathcal{B}(1 \; 2)$ and $j<l.$  We compute:
\begin{align*}
\xi_1^{-1}\eta_{1}\eta_{2}(\overline{2}_1\otimes 2_3 \otimes 2 \otimes 0)&=\xi_1^{-1}\eta_1(\overline{2}_1\otimes 2_3 \otimes \overline{3} \otimes 1)\\
		&=\xi_1^{-1}(\overline{2}_3\otimes \overline{3} \otimes \overline{3} \otimes 1)\\
		&=2_1 \otimes \overline{2}\otimes  \overline{3} \otimes \overline{3} \otimes 1
\end{align*}
Therefore $\mathcal{R}(b_1 \otimes b_2)=2_1 \otimes\begin{array}{|c|c|c|c|}\hline 1 & \overline{3} & \overline{3} & \overline{2}\\ \hline \end{array}.$
\subsection{Combinatorial $R$ matrix for $\mathcal{B}_\natural\otimes \mathcal{B}_1$}
Recall the $U_q'(G_2^{(1)})$-crystal $\mathcal{B}_\natural$ which is isomorphic as a $U_q(G_2)$-crystal to $\mathcal{B}(\Lambda_2)$.  Then, we have the following: 
\begin{proposition} The complete set of highest weight vectors and their weights is given below: 
\begin{center}
\begin{tabular}{|l|l|}
\hline
Highest weight vector in& Classical \\
$\mathcal{B}_\natural\otimes \mathcal{B}_1$ &weight\\
\hline \hline
$1'\otimes 1$ & $\Lambda_1 + \Lambda_2$\\
$1'\otimes 2_1$ & $2\Lambda_2$\\
$1' \otimes \varnothing$ & $\Lambda_2$\\
$1'\otimes \hat{0}$ & $\Lambda_2$\\
\hline
\end{tabular}
\begin{tabular}{|l|l|}
\hline
Highest weight vector in& Classical \\
$\mathcal{B}_1\otimes \mathcal{B}_\natural$ &weight\\
\hline \hline
$1\otimes 1'$ & $\Lambda_1 + \Lambda_2$\\
$1\otimes 3'$ & $2\Lambda_2$\\
$\varnothing \otimes 1'$ & $\Lambda_2$\\
$1\otimes 0'$ & $\Lambda_2$\\
\hline
\end{tabular}
\end{center}
\end{proposition}
\begin{proof}
As a $U_q(G_2)$-crystal $\mathcal{B}_\natural\otimes \mathcal{B}_1\cong \mathcal{B}(\Lambda_2)\otimes (\mathcal{B}(\Lambda_1)\oplus \mathcal{B}(0))$.  By lemma 1, the highest weight elements in $\mathcal{B}_\natural\otimes \mathcal{B}_1$ are those of the form $b \otimes b'$ where $b$ is highest weight in $\mathcal{B}(\Lambda_2)$ and $b'\in \mathcal{B}(\Lambda_1),\mathcal{B}(0)$ satisfies $\tilde{e}_i^{\delta_{i2}+1}b'=0$.  Similarly, the highest weight elements in $\mathcal{B}_1\otimes \mathcal{B}_\natural$ are those of the form $b \otimes b'$ where $b$ is highest weight in $\mathcal{B}(\Lambda_1),\mathcal{B}(0)$ and $b'\in \mathcal{B}(\Lambda_2)$ satisfies
$$\left \{
\begin{array}{ll}
\tilde{e}_i^{\delta_{i1}+1}b'=0&  \text{if } b=1,\\
\tilde{e}_i b'=0&  \text{if } b=\varnothing,\\
\end{array}
\right .
$$
for $i=1,2$.  The proposition follows by inspection of the crystal graphs of $\mathcal{B}_1$ and $\mathcal{B}_\natural.$
\end{proof}
\begin{proposition}
Let $\overline{\mathcal{R}}:\mathcal{B}_\natural \otimes \mathcal{B}_1\to \mathcal{B}_1\otimes \mathcal{B}_\natural$ be the unique crystal isomorphism.  Then $\overline{\mathcal{R}}$ is given on the highest weight elements by:
$$
\overline{\mathcal{R}}(1'\otimes b')=\left \{ \begin{array}{ll}
	1 \otimes 1', & \text{ if }b'=1,\\
	1 \otimes 3' & \text{ if }b'= 2_1,\\
	\varnothing \otimes 1' & \text{ if }b'=\varnothing,\\
	1\otimes 0' & \text{ if }b'=\hat{0}.
\end{array} \right .
$$
\end{proposition}
\begin{proof}
By property (2) of the definition of crystal isomorphism, we have the following: if $b\otimes b'\in \mathcal{B}_\natural\otimes \mathcal{B}_1$ and $\overline{\mathcal{R}}(b\otimes b')\in \mathcal{B}_1\otimes \mathcal{B}_\natural,$ then $\text{wt}(\overline{\mathcal{R}}(b\otimes b'))=\text{wt}(b\otimes b')$.  Therefore, since $1'\otimes 1$ and $1 \otimes 1'$ are the only elements of $\mathcal{B}_\natural \otimes \mathcal{B}_1$ and  $\mathcal{B}_1\otimes \mathcal{B}_\natural$ of weight $\Lambda_1+\Lambda_2$, we must have $\overline{\mathcal{R}}(1'\otimes 1)=1\otimes 1'$ and, similarly $\overline{\mathcal{R}}(1'\otimes 2_1)=1\otimes 3'$.  For $\overline{\mathcal{R}}(1'\otimes \varnothing),$ we compute 
\begin{eqnarray*}
\tilde{e}_2^2\tilde{f}_0\tilde{e}_1\tilde{e}_0(1'\otimes \varnothing)&=&1'\otimes 2_1\\
\overline{\mathcal{R}}(\tilde{e}_2^2\tilde{f}_0\tilde{e}_1\tilde{e}_0(1'\otimes \varnothing))&=&\overline{\mathcal{R}}(1'\otimes 2_1)\\
\tilde{e}_2^2\tilde{f}_0\tilde{e}_1\tilde{e}_0(\overline{\mathcal{R}}(1'\otimes \varnothing))&=&1\otimes 3'\\
\overline{\mathcal{R}}(1'\otimes \varnothing))&=&\tilde{f}_0\tilde{f}_1\tilde{e}_0\tilde{f}_2^2(1\otimes 3')=\varnothing \otimes 1'.\\
\end{eqnarray*}
Finally $\overline{\mathcal{R}}(1'\otimes \hat{0})=1\otimes 0'$, because it is the last remaining highest weight element.
\end{proof}
\section{Soliton Cellular Automaton}
We define $\mathcal{P}_L=\{b_1\otimes b_2 \otimes \cdots \otimes b_L\in\mathcal{B}_1^{\otimes L}| b_n=1, \text{ for } n \text{  sufficiently large}\}$ to be the set of \emph{states} of the $G_2^{(1)}$ soliton cellular automaton.  We depict the operation $\mathcal{R}(b\otimes b')=\tilde{b}'\otimes \tilde{b}$ by:
\begin{center}
\begin{tikzpicture}[scale=1]
\draw[->] (0,0) node[anchor=south]{$b'$} -- ++(0,-1) node[anchor=north]{$\tilde{b}'$};
\draw[->] (-.5,-.5) node[anchor=east]{$b$}-- ++(1,0)node[anchor=west]{$\tilde{b}$};
\end{tikzpicture}
\end{center}
 Fix $l>0$ and let $u_l\in \mathcal{B}_l$ be the element $(l,0,0,0,0,0)$.  For $p=b_1\otimes b_2 \otimes \cdots \otimes b_L\in \mathcal{P}_L$ we define the \emph{time evolution operator} $T_l(p)$:
\begin{equation}
T_l(p)\otimes u_l=\mathcal{R}_{L \: L+1}\mathcal{R}_{L-1 \: L}\cdots\mathcal{R}_{23} \mathcal{R}_{12}(u_l \otimes p). \label{ts}
\end{equation}
The transition of phase $T_l(b_1\otimes b_2 \otimes \cdots \otimes b_L)=\tilde{b}_1\otimes \tilde{b}_2 \otimes \cdots \otimes \tilde{b}_L$ is depicted as
\begin{center}
\begin{tikzpicture}[scale=1]
\draw[->] (0,0) node[anchor=south]{$b_1$} -- ++(0,-1) node[anchor=north]{$\tilde{b}_1$};
\draw[->] (-.5,-.5) node[anchor=east]{$u^{(0)}=u_l$}-- ++(1,0)node[anchor=west]{$u^{(1)}$};
\draw[->] (2,0) node[anchor=south]{$b_2$} -- ++(0,-1) node[anchor=north]{$\tilde{b}_2$};
\draw[->] (1.5,-.5) -- ++(1,0)node[anchor=west]{$u^{(2)}\cdots$};
\draw[->] (5.5,0) node[anchor=south]{$b_L$} -- ++(0,-1) node[anchor=north]{$\tilde{b}_L$};
\draw[->] (5,-.5)  node[anchor=east]{$u^{(L-1)}$} -- ++(1,0)node[anchor=west]{$u^{(L)}=u_l.$};
\end{tikzpicture}
\end{center}
We define the \emph{state energy} to be the sum:
\begin{equation}
E_l(p)=-\sum_{i=0}^{L-1} H(u^{(i)}\otimes b_{i+1}).
\end{equation}
We can also use the combinatorial $R$-matrix $\overline{\mathcal{R}}:\mathcal{B}_\natural \otimes \mathcal{B}_1 \to \mathcal{B}_1 \otimes \mathcal{B}_\natural$  to define an operator $T_\natural$ similar to $T_l$, by
	\begin{equation}
		T_\natural(p)\otimes b(p)=\overline{\mathcal{R}}_{L \: L+1}\overline{\mathcal{R}}_{L-1 \:L}
		\cdots\overline{\mathcal{R}}_{23}\overline{\mathcal{R}}_{12}(1' \otimes p). \label{tn}
	\end{equation}
In this case, $b(p)$ is dependent on the state $p$ so we indicate this dependence in the definition.
\subsection{$G_2^{(1)}$-solitons and their scattering rules}
Experience from other soliton cellular automata has shown that states $p\in\mathcal{P}_L$ satisfying $E_1(p)=1$ correspond to the so-called ``one-soliton states''. 
\begin{proposition}
In the $G_2^{(1)}$ soliton, $E_1(p)=1$ if and only if $p\neq 1^{\otimes L}$ and is one of the following:
\begin{itemize}
\item $1^{\otimes l}\otimes 3^{\otimes m} \otimes 2^{\otimes n} \otimes 1^{\otimes r}$,
\item $1^{\otimes l}\otimes 3^{\otimes m}\otimes 2_1 \otimes 2^{\otimes n} \otimes 1^{\otimes r-1}$,
\item $1^{\otimes l}\otimes 3^{\otimes m}\otimes 2_2 \otimes 2^{\otimes n} \otimes 1^{\otimes r-1}$,
\end{itemize}
for some $l,m,n,r \in \mathbb{Z}_{\geq 0}$ such that $l+m+n+r=L$.
\end{proposition}
\begin{proof}
Let $p=b_1\otimes b_2\otimes \cdots \otimes b_L$ be such that $E_1(p)=1$.  For $b\otimes b' \in \mathcal{B}_1 \otimes \mathcal{B}_1$ it is the case that $\mathcal{R}(b \otimes b')=b\otimes b'.$  Thus, $E_1(p)=-H(1\otimes b_1)-\sum_{i=1}^L H(b_i\otimes b_{i+1})=1$.  We compute $H(1\otimes 1)=0, H(1\otimes 2)=H(1\otimes 2_1)=H(1\otimes 2_2)=H(1\otimes 3)=H(1\otimes \varnothing)=-1$ and $H(1\otimes b)=-2$ in every other case.  Therefore, the first $b_i\neq 1$ must be either $2,2_1,2_2,3,$ or $\varnothing.$  However, $H(\varnothing \otimes b)>0$ for all $b\in \mathcal{B}_1,$ which rules out that case.  Similarly, we have $H(3\otimes 1)=H(3 \otimes 2)=H(3 \otimes 2_1)=H(3\otimes 2_2)=H(3 \otimes 3)=0,$  or $H(3 \otimes b)>0$ otherwise.  Therefore, only $1,2,2_1,2_2,$ or $3$ can follow $3$.  In turn, we see that only $2$ or $1$ can follow $2_1$ or $2_2$, and only $2$ or $1$ can follow 2.  Finally, a trailing 1 can only be followed by another 1, or $E_1(p)$ would be greater than 1.  Therefore, we see that if $E_1(p)=1$ then $p\neq 1^{\otimes L}$ and falls into one of the three cases given in the statement of the theorem.  The converse is easy to prove.
\end{proof}

\begin{proposition}
Let $p\in \mathcal{P}_L$.  If $\tilde{e}_2(p)\neq 0$ then $\tilde{e}_2T_l(p)=T_l(\tilde{e}_2(p))$ and $E_l(\tilde{e}_2(p))=E_l(p)$, otherwise $\tilde{e}_2T_l(p)=0$.  The same relations hold for $\tilde{f}_2$.  
\end{proposition}
\begin{proof}
It is easy to see that $\varepsilon_2(u_l)=\varphi_2(u_l)=0$, so by the tensor product rule $\tilde{f}_2(u_l\otimes p)=u_l\otimes \tilde{f}_2p$ and $\tilde{f}_2(T_l(p)\otimes u_l)=\tilde{f}_2T_l(p)\otimes u_l$.  Therefore, acting on both sides of \eqref{ts} by $\tilde{f}_2$ and commuting $\tilde{f}_2$ with $\mathcal{R}$ gives:
\begin{eqnarray*}
\tilde{f}_2T_l(p)\otimes u_l&=&\mathcal{R}_{L \: L+1}\mathcal{R}_{L-1 \: L}\cdots\mathcal{R}_{23} \mathcal{R}_{12}(u_l \otimes \tilde{f}_2p)\\
	&=&\left \{\begin{array}{l}T_l(\tilde{f}_2p)\otimes u_l, \text{ if }\tilde{f}_2p\neq0,\\
		0, \text{ otherwise.}
		\end{array}\right .
\end{eqnarray*}
Therefore, we see that $\tilde{f}_2T_l(p)=T_l(\tilde{f}_2p)$ if $\tilde{f}_2p\neq 0$ and $\tilde{f}_2T_l(p)=0$ otherwise.
Exactly the same argument works for the $\tilde{e}_2$ case (using the tensor product rule for $\tilde{e}_2$ instead of the one for $\tilde{f}_2)$.

Let $p=b_1\otimes b_2 \otimes \cdots \otimes b_L$, $p_0=z^0b_1\otimes z^0b_2\otimes \cdots \otimes z^0b_L\in \text{Aff}(\mathcal{B}_1)^{\otimes L}$, and $p'=\tilde{f}_2p.$  Then
\begin{eqnarray}
\lefteqn{\mathcal{R}_{L \: L+1}^{\text{Aff}}\mathcal{R}_{L-1 \: L}^{\text{Aff}}\cdots\mathcal{R}_{23}^{\text{Aff}} \mathcal{R}_{12}^{\text{Aff}}(z^0u_l \otimes p'_0)} \hspace{3cm} \nonumber\\
		&=&z^{H(u^{(0)}\otimes b_1')}\tilde{b}_1'\otimes z^{H(u^{(1)}\otimes b_2')}\tilde{b}'_2\otimes \cdots \nonumber\\
		&&\otimes \:z^{H(u^{(L-1)}\otimes b_L')}\tilde{b}_L'\otimes z^{-\sum_{i=0}^{L-1}H(u^{(i)}\otimes b_{i+1}')}u_l \label{en1}.
\end{eqnarray}
On the other hand,
\begin{eqnarray}
\lefteqn{\mathcal{R}_{L \: L+1}^{\text{Aff}}\mathcal{R}_{L-1 \: L}^{\text{Aff}}\cdots\mathcal{R}_{23}^{\text{Aff}} \mathcal{R}_{12}^{\text{Aff}}(z^0u_l \otimes p'_0)}\hspace{3cm}\nonumber\\
		&=&\mathcal{R}_{L \: L+1}^{\text{Aff}}\mathcal{R}_{L-1 \: L}^{\text{Aff}}\cdots\mathcal{R}_{23}^{\text{Aff}} \mathcal{R}_{12}^{\text{Aff}}(z^0u_l \otimes \tilde{f}_2p_0)\nonumber\\
		&=&\tilde{f}_2(\mathcal{R}_{L \: L+1}^{\text{Aff}}\mathcal{R}_{L-1 \: L}^{\text{Aff}}\cdots\mathcal{R}_{23}^{\text{Aff}} \mathcal{R}_{12}^{\text{Aff}}(z^0u_l \otimes p_0))\nonumber\\
		&=&\tilde{f}_2\big (z^{H(u^{(0)}\otimes b_1)}\tilde{b}_1\otimes z^{H(u^{(1)}\otimes b_2)}\otimes \cdots\nonumber \\
		&&\otimes \:z^{H(u^{(L-1)}\otimes b_L)}\tilde{b}_L\big )\otimes z^{-\sum_{i=0}^{L-1}H(u^{(i)}\otimes b_{i+1})}u_l\label{en2}.
\end{eqnarray}
Comparing the powers of $z$ in \eqref{en1} and \eqref{en2} we see that
\begin{eqnarray*}
-\sum_{i=0}^{L-1}H(u^{(i)}\otimes b'_{i+1})&=&-\sum_{i=0}^{L-1}H(u^{(i)}\otimes b_{i+1})\\
E_l(p')&=&E_l(p)\\
\end{eqnarray*}
therefore,
$$
E_l(\tilde{f}_2(p))=E_l(p).
$$
A similar argument shows that $E_l(\tilde{e}_2p)=E_l(p).$
\end{proof}
Recall the $A_1^{(1)}$-crystal $\tilde{\mathcal{B}}_{l}=\{(x_1,x_2)\in \mathbb{Z}_{\geq 0}| x_1+x_2=l\}$, where $(x_1,x_2)$ can be associated with the tableau: 
$$\begin{array}{l @{} l}\underbrace{\begin{array}{|c|c|c|c|}\hline 1 & 1 & \cdots & 1 \\ \hline\end{array}}_{x_1}& 
\underbrace{\begin{array}{c|c|c|c|}\hline  2 & 2 & \cdots & 2\\ \hline \end{array}}_{x_2}\end{array}.$$ 
\begin{proposition} \label{prop:tricky}
Let $b\in \tilde{\mathcal{B}}_{3l}$ and define a sequence $(b_1,b_2,\dots, b_l)\in \mathcal{B}_1^{l}$ by sending the following triples to the corresponding elements of $\mathcal{B}_1$:
$$ {\young(111)} \mapsto 2 \qquad {\young(112)} \mapsto 2_1 \qquad  {\young(122)} \mapsto 2_2  \qquad {\young(222)} \mapsto 3 .$$
Then the map $i_l:\tilde{\mathcal{B}}_{3l}\to \mathcal{B}_1^{\otimes l}$ given by $ i_l(b)=b_l\otimes b_{l-1} \otimes \cdots \otimes b_1$ satisfies the relations
$$i_l(\tilde{e}_1^Ab)= \tilde{e}_2 i_l(b), \qquad i_l(\tilde{f}_1^Ab)=\tilde{f}_2i_l(b).$$
\end{proposition}
\begin{proof}[Sketch of proof]
In $\mathcal{B}_1$
$$2\stackrel{2}{\to}2_1 \stackrel{2}{\to} 2_2 \stackrel{2}{\to} 3.$$
In $\tilde{\mathcal{B}_{3}}$
$$ {\young(111)} \stackrel{1}{\to}{\young(112)}\stackrel{1}{\to}{\young(122)}\stackrel{1}{\to}{\young(222)}. $$
Now use the definition of $\tilde{f}_2$ on the $l$ fold tensor product to obtain the result.
\end{proof}
\emph{Remark:} The above map is a bijection from $\tilde{\mathcal{B}}_{3l}$ to $\{p\in\mathcal{P}_L|E_1(p)=1\}$ by proposition 9.

A state of the following form is called and $m$-soliton state:
$$\dots [l_1]\dots\dots[l_2]\dots\cdots\dots[l_m]\dots$$
 where $l_1>l_2> \cdots > l_m$, $\dots[l]\dots$ denotes a local configuration of the form $3^{\otimes m} \otimes 2^{\otimes l-m}$,  $3^{\otimes m} \otimes 2_1\otimes 2^{\otimes l-m-1}$, or  $3^{\otimes m} \otimes  2_2\otimes 2^{\otimes l-m-1},m\geq 0$, and the $[l_i]$ are separated by sufficiently many 1's.

\begin{proposition}
Let $p$ be a one-soliton state of length $l$.  Then 
\begin{enumerate}
\item $E_k(p)=\min(k,l)$,
\item $T_k(p)$ is obtained by rightward shift by $\min(k,l)$ lattice steps.
\end{enumerate}
\end{proposition}
\begin{proof}
By applying $\tilde{e}_2$ sufficiently many times, $p$ becomes $2^{\otimes l} \otimes 1 \otimes \cdots \otimes 1$.  By theorems 2 and 3 we obtain:
\begin{eqnarray}
\mathcal{R}(\begin{array}{|c|c|}\hline  1^i & 2^{l-i} \\ \hline \end{array}\otimes 2)
	&=&1 \otimes \begin{array}{|c|c|}\hline  1^{i-1} & 2^{l-i+1} \\ \hline \end{array} \text{ if } i>0,\label{A1}\\
\mathcal{R}(\begin{array}{|c|}\hline   2^{l} \\ \hline \end{array}\otimes 2)
	&=&2 \otimes \begin{array}{|c|}\hline   2^{l} \\ \hline \end{array}\label{A2}\\
\mathcal{R}(\begin{array}{|c|c|}\hline  1^i & 2^{l-i} \\ \hline \end{array}\otimes 1)
	&=&2 \otimes \begin{array}{|c|c|}\hline  1^{i+1} & 2^{l-i-1} \\ \hline \end{array} \text{ if }i<l,\label{A3}\\
\mathcal{R}(\begin{array}{|c|}\hline   1^{l} \\ \hline \end{array}\otimes 1)
	&=&1 \otimes \begin{array}{|c|}\hline   1^{l} \\ \hline \end{array}\label{A4}
\end{eqnarray}
and
\begin{eqnarray}
H(\begin{array}{|c|c|}\hline  1^i & 2^{l-i} \\ \hline \end{array}\otimes 2)
	&=&-1 \text{ if }i>0,\\
H(\begin{array}{|c|}\hline   2^{l} \\ \hline \end{array}\otimes 2)
	&=&0 ,\\
H(\begin{array}{|c|c|}\hline  1^i & 2^{l-i} \\ \hline \end{array}\otimes 1)
	&=&0.
\end{eqnarray}
where, as before, the symbol $b^j$ means that $b$ is repeated $j$ times.  If $k<l$ then 
\begin{eqnarray*}
T_k(p)\otimes u_l&=&\mathcal{R}_{L\; L+1}\cdots \mathcal{R}_{23}\mathcal{R}_{12}(u_l\otimes 2^{\otimes l}\otimes 1^{L-l})\\
	&=& 1^{\otimes k}\otimes 2^{\otimes l} \otimes 1^{L-k-l} \otimes u_k,
\end{eqnarray*}
and
\begin{eqnarray*}
E_l(p)=-\sum_{i=0}^{L-1}H(u^{(i)}\otimes b_i)=k.
\end{eqnarray*}
Otherwise,
\begin{eqnarray*}
T_k(p)\otimes u_l&=&\mathcal{R}_{L\; L+1}\cdots \mathcal{R}_{23}\mathcal{R}_{12}( u_l\otimes 2^{\otimes l}\otimes 1^{L-L})\\
	&=& 1^{\otimes l}\otimes 2^{\otimes l} \otimes 1^{L-2l} \otimes u_k,
\end{eqnarray*}
and
\begin{eqnarray*}
E_l(p)=-\sum_{i=0}^{L-1}H(u^{(i)}\otimes b_i)=l.
\end{eqnarray*}
The result follows because $\tilde{f}_2$ commutes with $T_l$ and preserves $E_l$ (proposition 9).
\end{proof}

We now consider the two-soliton case $$p= \dots [l_1]\dots [l_2]\dots$$ where $l_1 > l_2$.  We can use proposition 10 to associate a two-soliton state $T_r^t(p):=p_t$ at time $t$ with the element $z^{-k_1}b_1 \otimes z^{-k_2}b_2\in \text{Aff}(\tilde{\mathcal{B}}_{3l_1})\otimes \text{Aff}(\tilde{\mathcal{B}}_{3l_2}),$ where $k_i:=-\min(r,l_i)t+\gamma_i$, where $\gamma_i$ is the number of `1's to the left of $[l_i]$.  If $r>l_2$, then we expect to see the longer soliton catch up with and collide with the shorter one, and, after sufficiently many time steps separate out into another two soliton state.  At the end of this section, we will prove that is the case.  In the following lemmas, we identify elements in  $\text{Aff}(\tilde{\mathcal{B}}_{3l_1})\otimes \text{Aff}(\tilde{\mathcal{B}}_{3l_2})$ with their corresponding two-soliton states.

\begin{lemma}[Analogous to lemma 4.15 in \cite{Yd}]
Suppose we have a one-soliton state $p=\dots [l]\dots$ corresponding to $z^{-k}(i,3l-i) \in \text{Aff}(\mathcal{B}_{3l}),0\leq i \leq 3l.$  Then $T_\natural(p)$ is another one-soliton state, corresponding to $z^{-k}(3l,0)$, if $i=3l$, and $z^{-k-1}(i+1,3l-1)$ otherwise.  We also have $b(p)=1'$ if $i=3l$ and $b(p)=2'$ otherwise. 
\end{lemma}
\begin{proof}
We ignore the initial `1's.
If $i=3l$ then $p=2^{\otimes l}\otimes 1^{\otimes L-l}$.  We compute $T_\natural(p)$ as follows:
\begin{center}
\begin{tikzpicture}[scale=1]

\draw[->] (4,0) node[anchor=south]{$2$} -- ++(0,-1) node[anchor=north]{$2$};
\draw[->] (3.75,-.5)node[anchor=east]{$1'$} -- ++(.5,0)node[anchor=west]{$1'$};
\draw[->] (5,0) node[anchor=south]{$2$} -- ++(0,-1) node[anchor=north]{$2$};
\draw[->] (4.75,-.5) -- ++(.5,0)node[anchor=west]{$1' \cdots$};

\draw[->] (7,0) node[anchor=south]{$2$} -- ++(0,-1) node[anchor=north]{$2$};
\draw[->] (6.75,-.5)  node[anchor=east]{$1'$} -- ++(.5,0)node[anchor=west]{$1'$};
\draw[->] (8,0) node[anchor=south]{$1$} -- ++(0,-1) node[anchor=north]{$1$};
\draw[->] (7.75,-.5) -- ++(.5,0)node[anchor=west]{$1'$};
\draw[->] (9,0) node[anchor=south]{$1$} -- ++(0,-1) node[anchor=north]{$1$};
\draw[->] (8.75,-.5) -- ++(.5,0)node[anchor=west]{$1' \cdots$};

\draw[->] (10.5,0) node[anchor=south]{$1$} -- ++(0,-1) node[anchor=north]{$1$};
\draw[->] (10.25,-.5) -- ++(.5,0)node[anchor=west]{$1'.$};
\end{tikzpicture}
\end{center}
Therefore $T_\natural(p)=2^{\otimes l}\otimes 1^{\otimes L-l}=z^{-k}(3l,0)$ and $b(p)=1'$ as required.

If $i=3r, 1\leq r<l$ then $p=3^{\otimes l-r} \otimes 2^{\otimes r}\otimes 1^{\otimes L-l}$.  We compute $T_\natural(p)$ as follows:
\begin{center}
\begin{tikzpicture}[scale=1]
\draw[->] (0,0) node[anchor=south]{$3$} -- ++(0,-1) node[anchor=north]{$1$};
\draw[->] (-.25,-.5) node[anchor=east]{$1'$}-- ++(.5,0)node[anchor=west]{$\overline{3}'$};
\draw[->] (1,0) node[anchor=south]{$3$} -- ++(0,-1) node[anchor=north]{$3$};
\draw[->] (.75,-.5) -- ++(.5,0)node[anchor=west]{$\overline{3}'\cdots$};

\draw[->] (3,0) node[anchor=south]{$3$} -- ++(0,-1) node[anchor=north]{$3$};
\draw[->] (2.75,-.5)  node[anchor=east]{$\overline{3}'$} -- ++(.5,0)node[anchor=west]{$\overline{3}'$};
\draw[->] (4,0) node[anchor=south]{$2$} -- ++(0,-1) node[anchor=north]{$2_2$};
\draw[->] (3.75,-.5) -- ++(.5,0)node[anchor=west]{$3'$};
\draw[->] (5,0) node[anchor=south]{$2$} -- ++(0,-1) node[anchor=north]{$2$};
\draw[->] (4.75,-.5) -- ++(.5,0)node[anchor=west]{$3' \cdots$};

\draw[->] (7,0) node[anchor=south]{$2$} -- ++(0,-1) node[anchor=north]{$2$};
\draw[->] (6.75,-.5)  node[anchor=east]{$3'$} -- ++(.5,0)node[anchor=west]{$3'$};
\draw[->] (8,0) node[anchor=south]{$1$} -- ++(0,-1) node[anchor=north]{$2$};
\draw[->] (7.75,-.5) -- ++(.5,0)node[anchor=west]{$2'$};
\draw[->] (9,0) node[anchor=south]{$1$} -- ++(0,-1) node[anchor=north]{$1$};
\draw[->] (8.75,-.5) -- ++(.5,0)node[anchor=west]{$2' \cdots$};

\draw[->] (11,0) node[anchor=south]{$1$} -- ++(0,-1) node[anchor=north]{$1$};
\draw[->] (10.75,-.5)node[anchor=east]{$2'$} -- ++(.5,0)node[anchor=west]{$2'$};
\end{tikzpicture}
\end{center}
Therefore $T_\natural(p)=1\otimes 3^{\otimes l-r-1}\otimes 2_2 \otimes 2^{\otimes r}\otimes 1^{L-l-1}=z^{-k-1}(3r+1,3l-3r-1)$ and $b(p)=2'$ as required.  The cases where $r=0,l$ are similar.

If $i=3r+1, 1\leq r<l$ then $p=3^{\otimes l-r-1} \otimes2_2 \otimes  2^{\otimes r}\otimes 1^{L-l-1}$.  We compute $T_\natural(p)$ as follows:
\begin{center}
\begin{tikzpicture}[scale=1]
\draw[->] (0,0) node[anchor=south]{$3$} -- ++(0,-1) node[anchor=north]{$1$};
\draw[->] (-.25,-.5) node[anchor=east]{$1'$}-- ++(.5,0)node[anchor=west]{$\overline{3}'$};
\draw[->] (1,0) node[anchor=south]{$3$} -- ++(0,-1) node[anchor=north]{$3$};
\draw[->] (.75,-.5) -- ++(.5,0)node[anchor=west]{$\overline{3}'\cdots$};

\draw[->] (3,0) node[anchor=south]{$3$} -- ++(0,-1) node[anchor=north]{$3$};
\draw[->] (2.75,-.5)  node[anchor=east]{$\overline{3}'$} -- ++(.5,0)node[anchor=west]{$\overline{3}'$};
\draw[->] (4,0) node[anchor=south]{$2_2$} -- ++(0,-1) node[anchor=north]{$3$};
\draw[->] (3.75,-.5) -- ++(.5,0)node[anchor=west]{$0'$};
\draw[->] (5,0) node[anchor=south]{$2$} -- ++(0,-1) node[anchor=north]{$2_1$};
\draw[->] (4.75,-.5) -- ++(.5,0)node[anchor=west]{$3'$};
\draw[->] (6,0) node[anchor=south]{$2$} -- ++(0,-1) node[anchor=north]{$2$};
\draw[->] (5.75,-.5) -- ++(.5,0)node[anchor=west]{$3' \cdots$};

\draw[->] (8,0) node[anchor=south]{$2$} -- ++(0,-1) node[anchor=north]{$2$};
\draw[->] (7.75,-.5)  node[anchor=east]{$3'$} -- ++(.5,0)node[anchor=west]{$3'$};
\draw[->] (9,0) node[anchor=south]{$1$} -- ++(0,-1) node[anchor=north]{$2$};
\draw[->] (8.75,-.5) -- ++(.5,0)node[anchor=west]{$2'$};
\draw[->] (10,0) node[anchor=south]{$1$} -- ++(0,-1) node[anchor=north]{$1$};
\draw[->] (9.75,-.5) -- ++(.5,0)node[anchor=west]{$2' \cdots$};

\draw[->] (12,0) node[anchor=south]{$1$} -- ++(0,-1) node[anchor=north]{$1$};
\draw[->] (11.75,-.5)node[anchor=east]{$2'$} -- ++(.5,0)node[anchor=west]{$2'$};
\end{tikzpicture}
\end{center}
Therefore $T_\natural(p)=1\otimes 3^{\otimes l-r-1}\otimes 2_1\otimes  2^{\otimes r}\otimes 1^{L-l-1}=z^{-k-1}(3r+2,3l-3r-2)$ and $b(p)=2'$ as required.  The cases where $r=0,l$ are similar.

If $i=3r+2, 1\leq r<l$ then $p=3^{\otimes l-r-1} \otimes 2_1 \otimes  2^{\otimes r}\otimes 1^{L-l-1}$.  We compute $T_\natural(p)$ as follows:
\begin{center}
\begin{tikzpicture}[scale=1]
\draw[->] (0,0) node[anchor=south]{$3$} -- ++(0,-1) node[anchor=north]{$1$};
\draw[->] (-.25,-.5) node[anchor=east]{$1'$}-- ++(.5,0)node[anchor=west]{$\overline{3}'$};
\draw[->] (1,0) node[anchor=south]{$3$} -- ++(0,-1) node[anchor=north]{$3$};
\draw[->] (.75,-.5) -- ++(.5,0)node[anchor=west]{$\overline{3}'\cdots$};

\draw[->] (3,0) node[anchor=south]{$3$} -- ++(0,-1) node[anchor=north]{$3$};
\draw[->] (2.75,-.5)  node[anchor=east]{$\overline{3}'$} -- ++(.5,0)node[anchor=west]{$\overline{3}'$};
\draw[->] (4,0) node[anchor=south]{$2_1$} -- ++(0,-1) node[anchor=north]{$3$};
\draw[->] (3.75,-.5) -- ++(.5,0)node[anchor=west]{$3'$};
\draw[->] (5,0) node[anchor=south]{$2$} -- ++(0,-1) node[anchor=north]{$2$};
\draw[->] (4.75,-.5) -- ++(.5,0)node[anchor=west]{$3'$};
\draw[->] (6,0) node[anchor=south]{$2$} -- ++(0,-1) node[anchor=north]{$2$};
\draw[->] (5.75,-.5) -- ++(.5,0)node[anchor=west]{$3' \cdots$};

\draw[->] (8,0) node[anchor=south]{$2$} -- ++(0,-1) node[anchor=north]{$2$};
\draw[->] (7.75,-.5)  node[anchor=east]{$3'$} -- ++(.5,0)node[anchor=west]{$3'$};
\draw[->] (9,0) node[anchor=south]{$1$} -- ++(0,-1) node[anchor=north]{$2$};
\draw[->] (8.75,-.5) -- ++(.5,0)node[anchor=west]{$2'$};
\draw[->] (10,0) node[anchor=south]{$1$} -- ++(0,-1) node[anchor=north]{$1$};
\draw[->] (9.75,-.5) -- ++(.5,0)node[anchor=west]{$2' \cdots$};

\draw[->] (12,0) node[anchor=south]{$1$} -- ++(0,-1) node[anchor=north]{$1$};
\draw[->] (11.75,-.5)node[anchor=east]{$2'$} -- ++(.5,0)node[anchor=west]{$2'$};
\end{tikzpicture}
\end{center}
Therefore $T_\natural(p)=1\otimes 3^{\otimes l-r-1}\otimes  2^{\otimes r+1}\otimes 1^{L-l-1}=z^{-k-1}(3r+3,3l-3r-3)$ and $b(p)=2'$ as required.  The cases where $r=0,l$ are similar.
\end{proof}
Now consider the two soliton state $z^{-k_1}(3l_1,0)\otimes z^{-k_2}(i, 3l_2-i)$.  Thanks to lemma 3, the action of $T_\natural$ has no effect on the first soliton, and therefore we have the following corollary.
\begin{corollary}[Analogous to lemma 4.15 in \cite{Yd}]
Suppose we have a two-soliton state $p=\dots [l_1]\dots \cdots \dots [l_2]\dots$ corresponding to $z^{-k_1}(3l_1,0)\otimes z^{-k_2}(i,3l_2-i) \in \text{Aff}(\mathcal{\tilde{B}}_{3l_1})\otimes \text{Aff}(\mathcal{\tilde{B}}_{3l_2}),0\leq i \leq 3l_2.$  Then $T_\natural(p)$ is another two-soliton state, corresponding to $z^{-k_1}(3l_1,0)\otimes z^{-k_2}(3l_2,0)$, if $i=3l_2$, and $z^{-k_1}(3l_1,0) \otimes z^{-k_2-1}(i+1,3l_2-1)$ otherwise.  We also have $b(p)=1'$ if $i=3l_1$ and $b(p)=2'$ otherwise. 
\end{corollary}
\begin{lemma}[Analogous to lemma 4.17 in \cite{Yd}]
Assume that $l_1>l_2$.  For $p=z^{-k_1}(3l_1,0)\otimes z^{-k_2}(i,3l_2-i)\in \text{Aff}(\mathcal{B}_{3l_1})\otimes \text{Aff}(\mathcal{B}_{3l_2}),0\leq i \leq 3l_2$ we have
\begin{enumerate}
\item $T_{\natural}(\hat{\mathcal{R}}^{\text{Aff}}(p))=\hat{\mathcal{R}}^{\text{Aff}}(T_{\natural}(p))$, and
\item $b(p)=b(\hat{\mathcal{R}}^{\text{Aff}}(p)),$
\end{enumerate}
where we use the shifted energy function $\tilde{H}=2l_2+\hat{H}$ in $\hat{\mathcal{R}}^{\text{Aff}}$.
\end{lemma}

\begin{proof}
\emph{Case 1:} $i=0$.  Then $p=z^{-k_1}(3l_1,0)\otimes z^{-k_2}(3l_2,0)$.  
Therefore,
\begin{eqnarray*}
\hat{\mathcal{R}}^{\text{Aff}}(p)&=&z^{-k_2+2l_2}(3l_2,0)\otimes z^{-k_1-2l_2}(3l_1,0),\\
T_\natural(\hat{\mathcal{R}}^{\text{Aff}}(p))&=&z^{-k_2+2l_2}(3l_2,0)\otimes z^{-k_1-2l_2}(3l_1,0)\text{ (by corollary 1)},
\end{eqnarray*}
and $b(\hat{\mathcal{R}}^{\text{Aff}}(p))=1'.$

On the other hand,
\begin{eqnarray*}
T_{\natural}(p)&=&z^{-k_1}(3l_1,0)\otimes z^{-k_2}(3l_2,0),\\
\hat{\mathcal{R}}^{\text{Aff}}(T_{\natural}(p))&=&z^{-k_2+2l_2}(3l_2,0)\otimes z^{-k_1-2l_2}(3l_1,0),
\end{eqnarray*}
and $b(p)=1'$.  Therefore $T_{\natural}(\hat{\mathcal{R}}^{\text{Aff}}(p))=\hat{\mathcal{R}}^{\text{Aff}}(T_{\natural}(p))$, which is (1), and $b(p)=b(\hat{\mathcal{R}}^{\text{Aff}}(p))$, which is (2).\\

\emph{Case 2:} $0<i\leq 3l_2$. 
Then
\begin{eqnarray*}
\hat{\mathcal{R}}^{\text{Aff}}(p)&=&z^{-k_2+2l_2-i}(3l_2,0)\otimes z^{-k_1-2l_2+i}(3l_1-i,i),\\
T_\natural(\hat{\mathcal{R}}^{\text{Aff}}(p))&=&z^{-k_2+2l_2-i}(3l_2,0)\otimes z^{-k_1-2l_2+i-1}(3l_1-i+1,i-1)\text{ (by corollary 1)},
\end{eqnarray*}
and $b(\hat{\mathcal{R}}^{\text{Aff}}(p))=2'.$
On the other hand,
\begin{eqnarray*}
T_{\natural}(p)&=&z^{-k_1}(3l_1,0)\otimes z^{-k_2-1}(3l_2-i+1,i-1)\text{ (by corollary 1)},\\
\hat{\mathcal{R}}^{\text{Aff}}(T_{\natural}(p))&=&z^{-k_2+2l_2-i}(3l_2,0)\otimes z^{-k_1-2l_2+i-1}(3l_1-i+1,i-1),
\end{eqnarray*}
and $b(p)=2'$.  Therefore $T_{\natural}(\hat{\mathcal{R}}^{\text{Aff}}(p))=\hat{\mathcal{R}}^{\text{Aff}}(T_{\natural}(p)),$ which is (1), and $b(p)=b(\hat{\mathcal{R}}^{\text{Aff}}(p)),$ which is (2).
\end{proof}
\begin{lemma}[Analogous to lemma 4.18 in \cite{Yd}]Let $p\in \mathcal{P}_L,l>0,L\gg 0.$  Then
\begin{enumerate}
\item$T_{\natural}(T_l(p))=T_l(T_\natural(p))$,
\item $b(T_l(p))\otimes u_l=\overline{\mathcal{R}}(u_l\otimes b(p))$.
\end{enumerate}
\end{lemma}
\begin{proof}
By definition
\begin{equation*}
T_\natural(T_l(p))\otimes b(T_l(p))=\overline{\mathcal{R}}_{L,L+1} \cdots \overline{\mathcal{R}}_{23}\overline{\mathcal{R}}_{12}(1'\otimes T_l(p)),
\end{equation*}
and
\begin{equation*}
T_l(p)\otimes u_l=\mathcal{R}_{L,L+1} \cdots\mathcal{R}_{23}\mathcal{R}_{12}(u_l\otimes p).
\end{equation*}
Now, combining the definitions, we compute:
\begin{eqnarray*}
\lefteqn{T_\natural(T_l(p))\otimes b(T_\natural(p))\otimes u_l}\\
&=&\overline{\mathcal{R}}_{L,L+1} \cdots \overline{\mathcal{R}}_{23}\overline{\mathcal{R}}_{12}\mathcal{R}_{L+1,L+2} \cdots\mathcal{R}_{34}\mathcal{R}_{23}(1'\otimes u_l\otimes p)\\
	&=&\overline{\mathcal{R}}_{L,L+1} \mathcal{R}_{L+1,L+2} \cdots \overline{\mathcal{R}}_{23}\mathcal{R}_{34}\overline{\mathcal{R}}_{12}\mathcal{R}_{23}(1'\otimes u_l\otimes p)\\
	&=&\overline{\mathcal{R}}_{L+1,L+2}\mathcal{R}_{L,L+1}\overline{\mathcal{R}}_{L+1,L+2}
\cdots\\ &&\qquad
\overline{\mathcal{R}}_{34}\mathcal{R}_{23}\overline{\mathcal{R}}_{34}
\overline{\mathcal{R}}_{23}^{-1}\overline{\mathcal{R}}_{23}\mathcal{R}_{12}\overline{\mathcal{R}}_{23}
\overline{\mathcal{R}}_{12}^{-1}(1'\otimes u_l\otimes p)\\
		&& \qquad \text{(by the Yang-Baxter equation)}\\
	&=&\overline{\mathcal{R}}_{L+1,L+2}\mathcal{R}_{L,L+1} \overline{\mathcal{R}}_{L+1,L+2} \cdots
		\mathcal{R}_{23}\overline{\mathcal{R}}_{34}\mathcal{R}_{12}\overline{\mathcal{R}}_{23}(u_l\otimes 1'\otimes p)\\
	&& \qquad \text{(since $\overline{\mathcal{R}}^{-1}(1'\otimes u_l)=u_l\otimes 1'$)}\\
	&=&\overline{\mathcal{R}}_{L+1,L+2}\mathcal{R}_{L,L+1} \cdots \mathcal{R}_{23}\mathcal{R}_{12}\overline{\mathcal{R}}_{L+1,L+2} \cdots
		\overline{\mathcal{R}}_{34}\overline{\mathcal{R}}_{23}(u_l\otimes 1'\otimes p)\\
	&=&\overline{\mathcal{R}}_{L+1,L+2}\mathcal{R}_{L,L+1} \cdots \mathcal{R}_{23}\mathcal{R}_{12}
		(u_l\otimes T_\natural(p)\otimes b(p))\\
	&=&\overline{\mathcal{R}}_{L+1,L+2}(T_l(T_\natural(p))\otimes u_l \otimes b(p))\\
	&=&T_l(T_\natural(p))\otimes \overline{\mathcal{R}}(u_l \otimes b(p)).
\end{eqnarray*}
Comparing both sides we see that $T_\natural(T_l(p))=T_l(T_\natural(p))$ and $b(T_\natural(p))\otimes u_l=\overline{\mathcal{R}}(u_l\otimes b(p))$.
\end{proof}
\begin{theorem}[Analogous to theorem 4.9 in \cite{Yd}] \label{th:main}
Let $p=\dots [l_1] \dots [l_2]\dots\in \mathcal{P}_L$ be a two-soliton state with $l_1>l_2$, corresponding to $z^{k_1}b_1\otimes z^{k_2}b_2 \in \text{Aff}(\mathcal{B}_{l_1})\otimes \text{Aff}(\mathcal{B}_{l_2})$ where $k_1,k_2\leq 0$.  Then after sufficiently many applications of $T_r, r>l_2$ the new state is given by 
$$\hat{\mathcal{R}}^{\text{Aff}}(z^{k_1}b_1\otimes z^{k_2}b_2)=z^{k_2'}\tilde{b}_2\otimes z^{k_1'}\tilde{b}_1$$ 
with phase shift
$$k_2'-k_2=k_1-k_1'=2l_2+\hat{H}(b_1\otimes b_2).$$
\end{theorem}
\begin{proof}
Let
$$p=z^{k_1}(x_1,x_2)\otimes z^{k_2}(y_1, y_2)\in \text{Aff}(\tilde{\mathcal{B}}_{3l_1})\otimes \text{Aff}(\tilde{\mathcal{B}}_{3l_2}).$$ 
By proposition 9, $T_r$ commutes with $\tilde{e}^A_1, \tilde{f}^A_1$.  Thus, it is enough to check the scattering rule for the highest weight elements 
$$
z^{k_1}(3l_1,0)\otimes z^{k_2}(y_1, y_2)\in \text{Aff}(\tilde{\mathcal{B}}_{3l_1})\otimes \text{Aff}(\tilde{\mathcal{B}}_{3l_2}).
$$
We will show the statement is true by induction on $y_2$.\\
Suppose $y_2=0$.  Then $y_1=3l_2.$  In this case, the cellular automaton time evolution is reduced to that of the $A_1^{(1)}$ cellular automaton by \eqref{A1}--\eqref{A4}:
\begin{eqnarray*}
\mathcal{R}(\begin{array}{|c|c|}\hline  1^i & 2^{l-i} \\ \hline \end{array}\otimes 2)
	&=&1 \otimes\begin{array}{|c|c|}\hline  1^{i-1} & 2^{l-i+1} \\ \hline \end{array} \text{ if } i>0,\\
\mathcal{R}(\begin{array}{|c|}\hline   2^{l} \\ \hline \end{array}\otimes 2)
	&=&2 \otimes \begin{array}{|c|}\hline   2^{l} \\ \hline \end{array}\\
\mathcal{R}(\begin{array}{|c|c|}\hline  1^i & 2^{l-i} \\ \hline \end{array}\otimes 1)
	&=&2 \otimes \begin{array}{|c|c|}\hline  1^{i+1} & 2^{l-i-1} \\ \hline \end{array} \text{ if }i<l,\\
\mathcal{R}(\begin{array}{|c|}\hline   1^{l} \\ \hline \end{array}\otimes 1)
	&=&1 \otimes \begin{array}{|c|}\hline   1^{l} \\ \hline \end{array}.
\end{eqnarray*}
Therefore, the scattering rule is the same as for the $A_1^{(1)}$ case, namely:
$$
T_r^{t}(p)=\hat{\mathcal{R}}^{\text{Aff}}(p),
$$
for sufficiently large $t$.

Assume the statement is true for $<y_2$.  By corollary 1, $T_\natural(p)=z^{k_1}(3l_1,0)\otimes z^{k_2-1}(y_1+1,y_2-1)$, so the inductive assumption holds.  Therefore, $T_r^t(T_\natural(p))$ is a 2-soliton state, for sufficiently large $t$, and 
$$
T_r^t(T_\natural(p))=\hat{\mathcal{R}}^{\text{Aff}}(T_\natural(p)).
$$
Therefore, by lemmas 4 (1) and 5 (1), we have
$$
T_\natural(T_r^t(p))=T_\natural(\hat{\mathcal{R}}^{\text{Aff}}(p)).
$$
Also, by lemma 5 (2), we have $b(\hat{\mathcal{R}}^{\text{Aff}}(p))=b(p),$ and by corollary 1 $b(p)=2'$ since $p$ is assumed to be a 2-soliton state and $y_2>0$.  By lemma 4 (2), $b(T_r(p))\otimes u_r=\overline{\mathcal{R}}(u_r\otimes b(p))=\overline{\mathcal{R}}(u_r\otimes 2')$.  But $\tilde{e}_2(u_r\otimes 2')=u_r\otimes 1'$, $\overline{\mathcal{R}}(u_r\otimes 1')=1'\otimes u_r,$ and $\tilde{f}_2(1'\otimes u_r)=2' \otimes u_r$.  Therefore $b(T_r(p))=2'$ and, by repeated application of lemma 4.2 (2), we have  $b(T_r^t(p))=2'$.  Therefore, $b(\hat{\mathcal{R}}^{\text{Aff}}(p))=b(T_r^t(p))=2'$, and $T_\natural(T_r^t(p))=T_\natural(\hat{\mathcal{R}}^{\text{Aff}}(p)).$  Therefore, the crystal isomorphisms in \eqref{tn} can be inverted to get
$$
T_r^t(p)=\hat{\mathcal{R}}^{\text{Aff}}(p)
$$
is a 2-soliton state.

\end{proof}
\begin{corollary}
Let $p=\dots[l_1]\dots[l_2]\dots \cdots \dots [l_m] \dots$ be an $m$-soliton state with $l_1>l_2>\cdots > l_m$.  Then the scattering of solitons is factorized into two-body scattering.
\end{corollary}
\subsection{Examples of scattering of $G_2^{(1)}$-solitons}
In this section, we give some examples of two and three soliton scattering generated by computer.  Here we take $l=10$ in $T_l$, and $L=55$ in $\mathcal{P}_L$.

\noindent\emph{Example 1:}
$$\begin{array}{ll}
t=0:&
3	2_2	2	2	1	1	1	1	2_1	2	1	1	1	1	1	1	1	1	1	1	1	1	1	1	1	1	1	1	1	1	1	1	1	1	1	1	1	1	1	1	1	1	1	1	1	1	1	1	1	1	1	1	1	1	1	\\
t=1:&
1	1	1	1	3	2_2	2	2	1	1	2_1	2	1	1	1	1	1	1	1	1	1	1	1	1	1	1	1	1	1	1	1	1	1	1	1	1	1	1	1	1	1	1	1	1	1	1	1	1	1	1	1	1	1	1	1	\\
t=2:&
1	1	1	1	1	1	1	1	3	2_2	2	1	2_1	2	2	1	1	1	1	1	1	1	1	1	1	1	1	1	1	1	1	1	1	1	1	1	1	1	1	1	1	1	1	1	1	1	1	1	1	1	1	1	1	1	1	\\
t=3:&
1	1	1	1	1	1	1	1	1	1	1	3	2_2	1	1	2_1	2	2	2	1	1	1	1	1	1	1	1	1	1	1	1	1	1	1	1	1	1	1	1	1	1	1	1	1	1	1	1	1	1	1	1	1	1	1	1	\\
t=4:&
1	1	1	1	1	1	1	1	1	1	1	1	1	3	2_2	1	1	1	1	2_1	2	2	2	1	1	1	1	1	1	1	1	1	1	1	1	1	1	1	1	1	1	1	1	1	1	1	1	1	1	1	1	1	1	1	1	
\end{array}$$
The initial state corresponds to $z^0(7,5)\otimes z^{-8}(5,1).$  After interacting, the two-soliton state that emerges is $z^{-5}(1,5)\otimes z^{-3}(11,1)=\hat{\mathcal{R}}^{\text{Aff}}(z^0(7,5)\otimes z^{-8}(5,1)),$ in agreement with theorem 4.

\noindent \emph{Example 2:}
$$\begin{array}{ll}
t=0:&
3	3	2_1	1	1	2_1	2	1	1	2_2	1	1	1	1	1	1	1	1	1	1	1	1	1	1	1	1	1	1	1	1	1	1	1	1	1	1	1	1	1	1	1	1	1	1	1	1	1	1	1	1	1	1	1	1	1	\\
t=1:&
1	1	1	3	3	2_1	1	2_1	2	1	2_2	1	1	1	1	1	1	1	1	1	1	1	1	1	1	1	1	1	1	1	1	1	1	1	1	1	1	1	1	1	1	1	1	1	1	1	1	1	1	1	1	1	1	1	1	\\
t=2:&
1	1	1	1	1	1	3	2_2	1	3	2	2_3	1	1	1	1	1	1	1	1	1	1	1	1	1	1	1	1	1	1	1	1	1	1	1	1	1	1	1	1	1	1	1	1	1	1	1	1	1	1	1	1	1	1	1	1\\
t=3:&
1	1	1	1	1	1	1	1	3	2_2	1	1	\overline{2}_3	2_1	2	1	1	1	1	1	1	1	1	1	1	1	1	1	1	1	1	1	1	1	1	1	1	1	1	1	1	1	1	1	1	1	1	1	1	1	1	1	1	1	1	\\
t=4:&
1	1	1	1	1	1	1	1	1	1	3	2_2	1	2_1	1	3	2_1	2	1	1	1	1	1	1	1	1	1	1	1	1	1	1	1	1	1	1	1	1	1	1	1	1	1	1	1	1	1	1	1	1	1	1	1	1	1	\\
t=5:&
1	1	1	1	1	1	1	1	1	1	1	1	3	2_2	2_1	1	1	1	3	2_1	2	1	1	1	1	1	1	1	1	1	1	1	1	1	1	1	1	1	1	1	1	1	1	1	1	1	1	1	1	1	1	1	1	1	1	\\
t=6:&
1	1	1	1	1	1	1	1	1	1	1	1	1	1	2_2	3	2_1	1	1	1	1	3	2_1	2	1	1	1	1	1	1	1	1	1	1	1	1	1	1	1	1	1	1	1	1	1	1	1	1	1	1	1	1	1	1	1	\\
t=7:&
1	1	1	1	1	1	1	1	1	1	1	1	1	1	1	2_2	1	3	2_1	1	1	1	1	1	3	2_1	2	1	1	1	1	1	1	1	1	1	1	1	1	1	1	1	1	1	1	1	1	1	1	1	1	1	1	1	1	\\
t=8:&
1	1	1	1	1	1	1	1	1	1	1	1	1	1	1	1	2_2	1	1	3	2_1	1	1	1	1	1	1	3	2_1	2	1	1	1	1	1	1	1	1	1	1	1	1	1	1	1	1	1	1	1	1	1	1	1	1	1	
\end{array}$$
The initial state corresponds to $z^0(2,7)\otimes z^{-5}(5,1) \otimes z^{-9}(1,2)$ and the final, separated, state corresponds to $z^{-8}(1,2)\otimes z^{-3}(2,4)\otimes z^{-3}\otimes (5,4)=\hat{\mathcal{R}}^{\text{Aff}}_{12}\hat{\mathcal{R}}^{\text{Aff}}_{23}\hat{\mathcal{R}}^{\text{Aff}}_{12}(z^0(2,7)\otimes z^{-5}(5,1) \otimes z^{-9}(1,2)),$ in agreement with corollary 2.

We now give an example which shows that it is necessary to have $l_1>l_2>l_3$ in corollary 2.\\
\emph{Example 3:}
$$\begin{array}{ll}
t=0:&
2	2	2	2	1	1	1	1	3	3	1	1	2	2	1	1	1	1	1	1	1	1	1	1	1	1	1	1	1	1	1	1	1	1	1	1	1	1	1	1	1	1	1	1	1	1	1	1	1	1	1	1	1	1	1	\\
t=1:&
1	1	1	1	2	2	2	2	1	1	3	3	1	1	2	2	1	1	1	1	1	1	1	1	1	1	1	1	1	1	1	1	1	1	1	1	1	1	1	1	1	1	1	1	1	1	1	1	1	1	1	1	1	1	1	\\
t=2:&
1	1	1	1	1	1	1	1	2	2	2	2	3	3	1	1	2	2	1	1	1	1	1	1	1	1	1	1	1	1	1	1	1	1	1	1	1	1	1	1	1	1	1	1	1	1	1	1	1	1	1	1	1	1	1	\\
t=3:&
1	1	1	1	1	1	1	1	1	1	1	1	2	2	\overline{3}	3	1	1	2	2	1	1	1	1	1	1	1	1	1	1	1	1	1	1	1	1	1	1	1	1	1	1	1	1	1	1	1	1	1	1	1	1	1	1	1	\\
t=4:&
1	1	1	1	1	1	1	1	1	1	1	1	1	1	1	1	\overline{3}	\overline{3}	1	1	2	2	1	1	1	1	1	1	1	1	1	1	1	1	1	1	1	1	1	1	1	1	1	1	1	1	1	1	1	1	1	1	1	1	1	\\
t=5:&
1	1	1	1	1	1	1	1	1	1	1	1	1	1	1	1	1	1	1	\overline{1}	1	1	2	2	2	2	1	1	1	1	1	1	1	1	1	1	1	1	1	1	1	1	1	1	1	1	1	1	1	1	1	1	1	1	1	\\
t=6:&
1	1	1	1	1	1	1	1	1	1	1	1	1	1	1	1	1	1	1	1	1	\overline{1}	1	1	1	1	2	2	2	2	1	1	1	1	1	1	1	1	1	1	1	1	1	1	1	1	1	1	1	1	1	1	1	1	1	\\
t=7:&
1	1	1	1	1	1	1	1	1	1	1	1	1	1	1	1	1	1	1	1	1	1	1	\overline{1}	1	1	1	1	1	1	2	2	2	2	1	1	1	1	1	1	1	1	1	1	1	1	1	1	1	1	1	1	1	1	1	
\end{array}$$
\emph{Remark:} Although separation does not occur in the previous example, if we were to apply the result of corollary 2 to this example we would expect to get the following:
\begin{eqnarray*}
\lefteqn{\hat{\mathcal{R}}^{\text{Aff}}_{23}\hat{\mathcal{R}}^{\text{Aff}}_{12}(z^0(12,0)\otimes z^{-8}(0,6)\otimes z^{-12}(6,0))}\\
	&=&\hat{\mathcal{R}}_{23}^{\text{Aff}}(z^{-10}(6,0)\otimes z^{2}(6,6)\otimes z^{-12}(6,0))\\
	&=&z^{-10}(6,0)\otimes z^{-8}(0,6)\otimes z^{-2}(12,0).
\end{eqnarray*}
The reason for the observed lack of separation is perhaps due to the fact that the leftmost two solitons in the initial state are shifted \emph{oppositely} to each other, thus creating an overlapping state.
\section{Acknowledgments}
MO is partially supported by the Grants-in-Aid for Scientific Research No. 23340007 and No. 23654007 from JSPS. EAW was supported by the NSF EAPSI/JSPS Summer Institutes grant.


\begin{thebibliography}{99}
\normalsize{
\bibitem{D} \textsc{V. Drinfel$'$d,} Hopf algebras and the quantum Yang-Baxter equation, \emph{Soviet Math. Dokl.} \textbf{32(1)} (1985) 254--258.

\bibitem{FOY} \textsc{K. Fukuda, M. Okado, Y. Yamada,} Energy functions in box-ball systems, \emph{Int. J. Mod. Phys.} \textbf{15(9)} (2000) 1379--1392.

\bibitem{HKT} \textsc{G. Hatayama, A. Kuniba, T. Takagi,} Soliton cellular automata associated with crystal bases, \emph{Nucl. Phys. B} \textbf{577}[PM] (2000) 619--645.

\bibitem{HKOTY} \textsc{G. Hatayama, A. Kuniba, M. Okado, T. Takagi, Y. Yamada,} Scattering rules on soliton cellular automata associated with crystal bases, \emph{Cont. Math.} \textbf{297} (2002) 151--182.

\bibitem{HK} \textsc{J. Hong, S.-J. Kang,} \emph{Introduction to Quantum Groups and Crystal Bases,} Graduate Studies in Mathematics, vol. 42, American Mathematical Society, 2002.

\bibitem{J} \textsc{M. Jimbo,} A $q$-difference analog of $U(\mathfrak{g})$ and the Yang-Baxter equation, \emph{Lett. Math. Phys.} \textbf{10(1)} (1985) 63--69.

\bibitem{K} \textsc{V. G. Kac,}  \emph{Infinite Dimensional Lie Algebras, $3^{\text{rd}}$ edition},  Cambridge University Press, 1990.

\bibitem{KKM}  \textsc{S.-J. Kang, M. Kashiwara, K. C. Misra,} Crystal base of Verma modules for quantum affine Lie algebras, \emph{Compositio Math.} \textbf{92} (1994) 299--325.

\bibitem{KMN} \textsc{S.-J. Kang, M. Kashiwara, K. C. Misra, T. Miwa, T. Nakashima, A. Nakayashiki,} Affine crystals and vertex models, \emph{Int. J. Mod. Phys.} \textbf{A7} (suppl. 1A) (1992) 449--484.

\bibitem{KM}  \textsc{S.-J. Kang, K. C. Misra,}  Crystal Bases and Tensor Product Decompositions of $U_q(G_2)$-Modules,  \emph{J. Algebra} \textbf{163} (1994) 675--691. 

\bibitem{Ka} \textsc{M. Kashiwara,}  Crystallizing the $q$-analog of universal enveloping algebras, \emph{Comm. Math. Phys.} \textbf{133} (1990) 249--260.

\bibitem{KMOY} \textsc{M. Kashiwara, K. C. Misra,  M. Okado, D. Yamada,} Perfect crystals of $U_q(D_4^{(3)}),$ \emph{J. Algebra} \textbf{317} (2007) 392--423.

\bibitem{L} \textsc{C. Lecouvey,} Combinatorics of Crystal Graphs for the Root Systems of Types $A_n, B_n, C_n, D_n,$ and $G_2$, \emph{MSJ Memoirs} \textbf{17} (2007) 11--41.

\bibitem{Lu} \textsc{G. Luzstig,} Canonical bases arising from quantized universal enveloping algebras, \emph{J. Amer. Math. Soc.} \textbf{3(2)} (1990) 447--498.

\bibitem{MMO}  \textsc{K. C. Misra, M. Mohamad, M. Okado,} Zero Action on Perfect Crystals for $U_q(G_2^{(1)}),$ \emph{SIGMA} \textbf{201}(2010) 022, 12 pages. 

\bibitem{M} \textsc{M. Mohamad,} preprint.

\bibitem{T}\textsc{D. Takahashi,} On some soliton systems defined by using boxes and balls, in \emph{Proceedings of the International Symposium on Nonlinear Theory and Its Applications (NOLTA '93)}, 1993, pp. 555--558.

\bibitem{TS} \textsc{D. Takahashi, J. Satsuma,} A soliton cellular automaton, \emph{J. Phys. Soc. Jpn.} \textbf{59} (1990) 3514--3519.

\bibitem{Yd} \textsc{D. Yamada,} Scattering rule in soliton cellular automata associated with crystal base of $U_q(D_4^{(3)}),$  \emph{J. Math. Phys.} \textbf{48} 043509 (2007) 1--28. 

\bibitem{Yn}  \textsc{S. Yamane,} Perfect crystals of $U_q(G_2^{(1)}),$ \emph{J. Algebra} \textbf{210} (1998) 440--486.
}
\end{thebibliography}
\end{document}